\title{The distance from a point to its opposite along the surface
of a box}
\author{S.~ Michael Miller\thanks{
University of California, Los Angeles}
        \and Edward F.~Schaefer
        \thanks{
        Santa Clara University
                  }\thanks{This work was
                  supported by Pennello Funds from our department.
The authors are grateful to Richard A. Scott for
helpful advice on defining the equivalence relation and John Girgis for teaching us how to create
the figures. The authors are also grateful to the referee for helpful
suggestions on the text as well as Figure~\ref{Flat}.}
        }

\documentclass{article}
\pdfoutput=1
\usepackage{graphicx}
\usepackage{amscd}
\usepackage{amssymb}
\DeclareGraphicsExtensions{.pdf,.png,.jpg}
\newtheorem{theorem}{Theorem}
\newtheorem{lemma}{Lemma}
\newtheorem{prop}{Proposition}
\newtheorem{cor}{Corollary}

\newenvironment{proof}{{\sc Proof:}}{~\hfill QED}
\newenvironment{AMS}{}{}
\newenvironment{keywords}{}{}
\newcommand{\ra}{{\rightarrow}}

\newcommand{\calF}{{\mathcal F}}
\newcommand{\calR}{{\mathcal R}}

\begin{document}
\newpage
\maketitle
\begin{abstract}
 Given a point (the ``spider'') on a rectangular box, we would like
to find  the minimal distance along the surface to
its opposite point (the ``fly'' - the reflection of the spider
across the center of the box). Without loss of generality, we can
assume that the box has dimensions $1\times a\times b$ with the
spider on one of the $1\times a$ faces (with $a\leq 1$). The
shortest path between the points is always a line segment for some
planar flattening of the box by cutting along edges. We then
partition the $1\times a$ face into regions, depending on which
faces this path traverses. This choice of faces determines an
algebraic distance formula in terms of $a$, $b$, and suitable
coordinates imposed on the face. We then partition the set of
pairs $(a,b)$ by homeomorphism of the borders of the $1\times a$
face's regions and a labeling of these regions.\end{abstract}




\section{Introduction}

In 1903, Henry Dudeney, a popular creator of mathematical puzzles,
posed the famous spider and fly problem in \cite{Du}: given a
spider and fly in a $30\times 12 \times 12$ foot room, the spider
on one $12\times 12$ wall, one foot below the ceiling and
equidistant from the sides, and the fly on the other $12\times 12$
wall, one foot above the floor and equidistant from the sides -
what is the shortest path the spider can take to reach the fly by
crawling along the walls, floor, and ceiling of the room? The most
obvious path, going straight up, then straight across the ceiling,
then straight down to the fly, is 42 feet long. The spider's
optimal path of $40$ feet requires it traverse five faces of the
room before reaching the spider. We can cut along certain edges of
the room, flatten out the room and then this path is a straight
line segment.

The problem can be generalized to an arbitrary point on an
arbitrarily sized rectangular box. By scaling and rotating, we can restrict
the dimensions of the box to $1\times a \times b$, with $0<a\leq
1$, $0<b$, and the spider to be a point on the $1\times a$ face.
 We wish to find the shortest distance along the surface of
the box (the path is called a geodesic) between this point and its
opposite - the point obtained by reflecting the original point
across the center of the box (this is the antipodal map).

We can assign $(x,y)$-coordinates to the points of the $1\times a$
side so it is described by $-\frac{a}{2}\leq x\leq \frac{a}{2}$
and $-\frac{1}{2}\leq y\leq \frac{1}{2}$. By symmetry, it suffices
to solve the problem for points in the fundamental region $\calF$
given by $0\leq x\leq \frac{a}{2}$ and $0\leq y\leq \frac{1}{2}$.
Dudeney's spider has $(a,b,x,y)=(1,\frac{5}{2},0,\frac{5}{12})$.
In \cite[p.\ 144 - 146]{Ra}, Ransom made progress on this question
for $x=0$ and $y=0$.

In this article, we attempt to solve the generalized problem.
One might think this has an easy, elegant solution; let us surprise you with how baroque the details actually become. There is a question of what form the solution should take.
The solution we found most
compelling is the following: for
each point $(a,b)$ in the $ab$-plank ($0 < a \leq 1, 0 < b$) and each point in the
associated $\calF$, we will describe six paths, at least one of
which must be the shortest.
We note that sometimes, for each of two nearby points $(a,b)$, the subsets
of the fundamental region,  on which
each of the six distance functions is smallest, ``look essentially the same''.
This is a topological notion; so we use topology to describe
an equivalence relation.
Let ${\mathcal F}_Q$ and ${\mathcal
F}_R$ be the fundamental regions associated to the points
$Q$ and $R$
 in the $ab$-plank. We consider
$Q$ and $R$ to be in the same equivalence class if
we can continuously deform the boundary curves (between regions on
which a given distance function is smallest) and edges of
${\mathcal F}_Q$ to the boundary curves and edges  of ${\mathcal
F}_R$ such that the shortest of the six paths associated to the points
bounded by these curves remain the same (this will be defined
precisely in Section~\ref{equrel} using homeomorphisms). We will
describe all 47 of the equivalence classes and the associated
$\calF$'s.

The computational proofs of the validity of the equivalence classes are
quite long and are presented in Appendix.

\section{The paths}
\label{The paths}

We consider the $1\times a$ face to be the top face of the
box, the face opposite it to be the bottom face of the box and the four
other faces to be side faces. As we look down on the box,
we orient the $1\times a$ face as we standardly orient the $xy$-plane.
If the path (starting in $\calF$) leaves
 the $1\times a$ face and then enters the side face to the right, up
above, left or down below then we denote the path
$p_{R,j}$, $p_{U,j}$, $p_{L,j}$ or $p_{D,j}$, respectively
(with j to be defined next).
If the path continues immediately to the bottom face,
then we let j = 0. This path crosses three faces.
If, after entering the first side face,
the path continues
in a clockwise direction (viewed from above)
and crosses a total of $n+1$ side faces beforing
entering the bottom face, then we let $j=n$. This path crosses $n+3$ faces.
If, after entering the first side face,
the path continues
in a counterclockwise direction and crosses a total of $n+1$ side faces beforing
entering the bottom face, then we let $j=-n$.

In Figure~\ref{Flat}, the six rectangles with all solid edges give one flattening
of the box. Other rectangles are faces for other flattenings on which
certain paths are line segments. The figure indicates paths $p_{R,0}$,
$p_{R,1}$, $p_{R,2}$ and $p_{U,-1}$. The fundamental region is shown
in the center of the figure;  its reflections by the antipodal map are indicated
near the boundary.

\begin{figure}
\centering
\includegraphics[width=7cm]{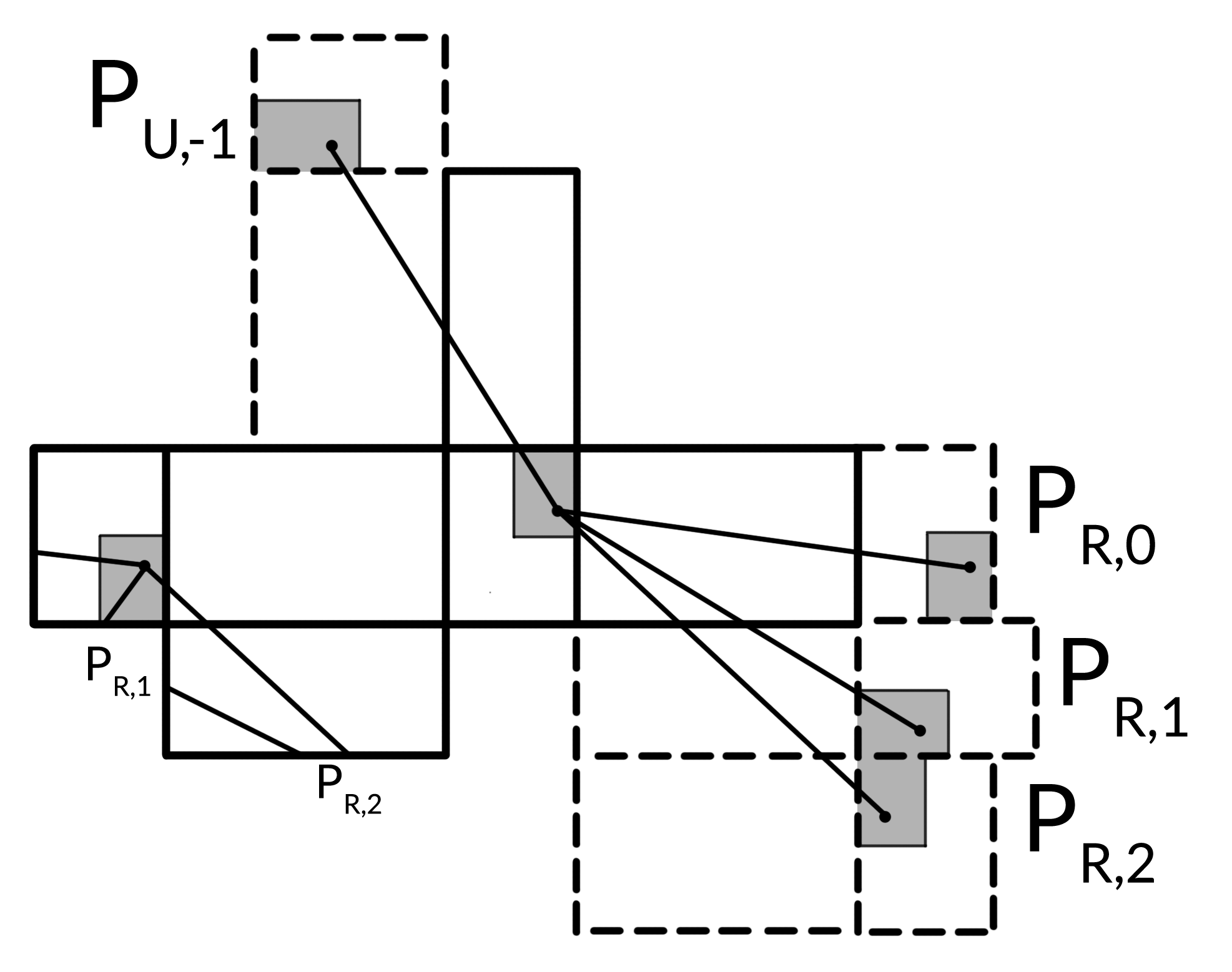}
\caption{Four paths on flattenings of the box}
\label{Flat}
\end{figure}

\begin{prop}
Paths with $|j|\geq 3$
can not be shortest.
\end{prop}

\begin{proof}
First assume $j=3$.
We will prove this by contradiction.
Assume, for some $(a,b,x,y)$ and some $\ast\in \{ R, U, L, D\}$,
that $d_{\ast,3} < d_{\ast,-1}$.
Note that $p_{\ast,3}$ and $p_{\ast,-1}$ cross the same initial side
faces and same
final side faces.
Consider a new path that is identical
to $p_{\ast,3}$ on the top and bottom faces.
Along the two side faces that $p_{\ast,-1}$ crosses, complete
the new path by connecting
the points where $p_{\ast,3}$ meets the edges of the top and bottom
faces with the shortest possible path (it will be a line segment when
those two side faces are flattened).
Note that the length of the new path is at least $d_{\ast,-1}$,
since it crosses the same faces as $p_{\ast,-1}$.

The path $p_{\ast,3}$ and the new path are identical except
on the sides faces. On flattened
side faces, $p_{\ast,3}$ and the new path
are each the  hypotenuse
of a right triangle with one side of length $b$ (the height of each side
face).
For $p_{\ast,3}$, the other side of the right triangle has length
greater than $1+a$,
which is the width of the second and third of the four side faces it
crosses.
On the new path, the other side of the right triangle has length
less  than $1+a$, which is the width of the two side faces it crosses.
So the new path has length less than $d_{\ast,3}$
and at least $d_{\ast,-1}$; this is a contradiction.
For the argument for $j=-3$, just flip the signs in the subscripts.

The cases where $|j|\geq 4$, where paths cross at least five side
faces, are obvious.
\end{proof}

Note that if we started at the point
$(x,y)=(\frac{a}{2},\frac{1}{2})$  and traveled along the edge not
bordering $\calF$, or had a path that eventually traveled along
this edge or any edge parallel to it, then this would be a degenerate form of more than
one of the paths given and thus would be considered in our
analysis.

By cutting along edges, the 3, 4 or 5 faces that a path crosses
can be
flattened out onto a plane, which is the easiest way to
determine the length of the corresponding paths.

Note that a line
segment path from one point to its opposite can leave and reenter
a flattening. We do not need to worry about this case because it
leaves and reenters on what is the same edge of the box. By gluing
those edges back together and possibly perturbing the path, we get a different and shorter path.

The referee pointed out that the map that reflects points across
the origin (the antipodal map) interchanges pairs of paths. The pairs
are $\{ p_{\ast,2},p_{\ast,-2}\}$ (for each $\ast$),
$\{ p_{R,1}, p_{U,-1}\}$, $\{ p_{U,1}, p_{L,-1}\}$,
$\{ p_{L,1}, p_{D,-1}\}$,  $\{ p_{D,1}, p_{R,-1}\}$,
$\{ p_{R,0}, p_{L,0}\}$, and $\{ p_{U,0}, p_{D,0}\}$.
This is illustrated in Figure~\ref{Flat} for $p_{R,1}$ and $p_{U,-1}$.
The length of each path in a pair is the same. So to minimize distance, it suffices to consider one
path in each pair. We choose $p_{\ast,2}$, $p_{R,1}$, $p_{U,1}$, $p_{L,1}$,
$p_{D,1}$, $p_{R,0}$, and $p_{U,0}$.
In
Table~\ref{10distances}, for a point $(x,y)\in \calF$, we give the
coordinates of the opposite point, given the obvious flattening for each of the 10 paths (this is illustrated in Figure~\ref{Flat} for four of the paths),
and $d_{\ast,j}$, the square of the distance between them. In the table we let
$a_1 = \frac{a+1}{2}$.

\begin{table}

 {\footnotesize
\begin{tabular}{l|l|l}
path & point opposite of $(x,y)$ & $d_{*,j}=$ squared distance to $(x,y)$ \\

\hline
$p_{R,0}$ & $(x+a+b,-y)$ & $4y^2 + a^2 + 2ba + b^2$ \\
$p_{R,1}$ & $(-y+a_1+b,-x-a_1)$ &  $2x^2 + 4xy + 2y^2 - 2bx  - 2by +  \frac{1}{2} a^2 + (b+1)a + b^2 + b +  \frac{1}{2}$ \\
$p_{R,2}$ & $(-x+a+b,y-a-1)$ & $4x^2 + (-4a - 4b)x + 2a^2 + (2b+2)a + b^2 + 1$ \\

 \hline
$p_{U,0}$ & $(-x,y+b+1)$ & $4x^2 + b^2 + 2b + 1$ \\
$p_{U,1}$ & $(y+a_1,x+a_1+b)$ &  $2x^2  -4xy + 2y^2 + 2bx  - 2by +  \frac{1}{2} a^2 + (b+1)a + b^2 +  b +  \frac{1}{2}$ \\
$p_{U,2}$ & $(x+a+1,-y+b+1)$ & $4y^2 + (-4b - 4)y + a^2 + 2a+ b^2 + 2b + 2$ \\

\hline
$p_{L,1}$ &  $(-y-a_1-b,-x+a_1)$ &  $2x^2 + 4xy + 2y^2 + 2bx + 2by +  \frac{1}{2} a^2 + (b + 1)a + b^2 + b +  \frac{1}{2}$ \\
$p_{L,2}$ &$(-x-a-b,y+a+1)$ & $4x^2 + (4a + 4b)x + 2a^2 + (2b + 2)a + b^2 + 1$ \\

\hline
$p_{D,1}$ & $(y-a_1,x-a_1-b)$ & $2x^2 -4xy + 2y^2 - 2bx  + 2by +  \frac{1}{2} a^2 + (b + 1)a + b^2 + b +  \frac{1}{2}$ \\
$p_{D,2}$ & $(x-a-1,-y-b-1)$ &  $4y^2 + (4b + 4)y + a^2 + 2a + b^2 + 2b + 2$ \\
\end{tabular} }
\caption{The 10 distance functions (note $a_1 = \frac{a+1}{2}$)}
\label{10distances}
\end{table}

By
inspection we see that $d_{R,2}\leq d_{L,2}$, $d_{U,2}\leq
d_{D,2}$ and $d_{D,1}\leq d_{L,1}$. Since we are minimizing
distance, we no longer consider $d_{L,2}$, $d_{D,2}$ or $d_{L,1}$.
So it suffices to consider the seven distance functions $d_{R,0}$,
$d_{R,1}$, $d_{R,2}$, $d_{U,0}$, $d_{U,1}$, $d_{U,2}$ and
$d_{D,1}$.

\begin{prop}
For all $P=(a,b,x,y)$ with $0< a \leq 1$, $0 < b $, $0 \leq x\leq
\frac{a}{2}$, $0\leq y\leq \frac{1}{2}$ we have $d_{D,1}(P) \geq
{\rm min} \{ d_{R,0}(P), d_{R,1}(P), d_{U,1}(P)\}$.
\end{prop}

\begin{proof}
We have $d_{D,1} - d_{U,1}=4b(y-x)$. So for $y\geq x$ we see
$d_{U,1}\leq d_{D,1}$. We have $d_{D,1} - d_{R,1} = 4y(b-2x)$. So
for $b\geq a$ we see $b\geq a\geq 2x$ and $d_{R,1}\leq d_{D,1}$.
Let $\calR$ be the subset $ b < a$ of the $ab$-plank and
$\calF'$ be the region $ y < x$ in the
fundamental region associated to a given $(a,b)\in \calR$. For
$d_{D,1}$ to be smallest (i.e.\ uniquely minimal) among the seven
distance functions, we need $(a,b)\in \calR$ and $(x,y)\in
\calF'$.

We now show that for all $(a,b,x,y)$ with $(a,b)\in \calR$ and
$(x,y)\in \calF'$ we have $d_{D,1}(a,b,x,y) \geq
d_{R,0}(a,b,x,y)$.  Note that for $(a,b,x,y)=(.8,.7,.3,.2)$ we have $d_{D,1}
> d_{R,0}$.
Since $d_{D,1}-d_{R,0}$ is a polynomial function, we can use a
continuity argument to show that if there is an $(a,b)\in \calR$
for which there is an $(x,y)\in \calF'$ such that
$d_{D_1}(a,b,x,y) < d_{R,0}(a,b,x,y)$ then the hyperbola
$d_{D,1}=d_{R,0}$ (for the given $(a,b)$) must pass through the
interior of $\calF'$, and so intersect the boundary of $\calF'$ in
two different points. We will show that this does not occur.

Let $\bar{\calF'}$ denote the closure of $\calF'$ in $\calF$.
The boundary of $\calF'$ (and of $\bar{\calF'}$) consists of $y=0$ and $y=x$ with $0\leq
x\leq \frac{a}{2}$, and $x=\frac{a}{2}$ with $0\leq y\leq
\frac{a}{2}$. The hyperbola $d_{D,1}=d_{R,0}$ meets $y=0$ where
$2x=b\pm \sqrt{(a+b-1)^2-2}$. Since $0 < b < a \leq 1$ we have $-1
< a+b-1 < 1$. So $(a+b-1)^2 - 2 < -1$. Thus the hyperbola does not
meet $y=0$ for $(a,b)\in \calR$. The hyperbola $d_{D,1}=d_{R,0}$
meets $y=x$ where $x^2 = \frac{1}{8}(-a^2-2ab+2a+2b+1)$. Since
$a\leq 1$ we have $3a^2+2ab \leq 3a+2b\leq 2a+2b+1$. So $2a^2 \leq
-a^2-2ab+2a+2b+1$ and $\frac{a^2}{4} \leq x^2$. Since $(x,y)\in
\bar{\calF'}$ we have $x^2\leq \frac{a^2}{4}$. So the hyperbola
$d_{D,1}=d_{R,0}$ can only meet $y=x$ in $\bar{\calF'}$ at
$(x,y)=(\frac{a}{2},\frac{a}{2})$. The hyperbola $d_{D,1}=d_{R,0}$
meets $x=\frac{a}{2}$ where $2y= b-a\pm
\sqrt{a^2-6ab+b^2+2a+2b+1}$. Since $b < a$ in  $\calR$ and $y \geq
0$, the hyperbola $d_{D,1}=d_{R,0}$ meets $x=\frac{a}{2}$ where
$2y= b-a+\sqrt{a^2-6ab+b^2+2a+2b+1}$. From above $3a^2+2ab\leq
2a+2b+1$; so $(2a-b)^2 \leq a^2-6ab+b^2+2a+2b+1$. We have $2a-b
\leq \sqrt{a^2-6ab+b^2+2a+2b+1}$. Thus $a \leq
b-a+\sqrt{a^2-6ab+b^2+2a+2b+1}$. So $a \leq 2y$. In $\bar{\calF'}$ we
have $y \leq x \leq \frac{a}{2}$. So the hyperbola $d_{D,1}=d_{R,0}$
can only meet $x=\frac{a}{2}$ in $\bar{\calF'}$ at
$(x,y)=(\frac{a}{2},\frac{a}{2})$. So the hyperbola can not meet
the boundary of $\calF'$ in two different points.
\end{proof}

So we see that the remaining six distance functions $d_{R,0}$,
$d_{R,1}$, $d_{R,2}$, $d_{U,0}$, $d_{U,1}$ and $d_{U,2}$ are
sufficient.
We will see they are also necessary, in the sense that
there are points $(a,b)$ in the $ab$-plank and points $(x,y)$
(in the associated fundamental regions) for which each of the six
distance functions is strictly smaller than the other five.

For the remainder of this article, for a given $(a,b)$, we say
that one of those six distance functions is {\it smallest}
(respectively {\it minimal}), for a given $(x,y)$, if it is
strictly smaller than (respectively smaller than or equal to) the
other five distance functions. Note that the regions on which a
distance function is smallest (respectively minimal) are open
(respectively closed) subsets of $\calF$.

\section{The equivalence relation}
\label{equrel}

For each $(a,b)$, there is an associated fundamental region
$\calF$. For each of the six distance functions, we can find the
subset of $\calF$ on which the distance function is smallest. Note
that where two of these subsets border, two or more distance
functions will have identical values.
As suggested in the
introduction, we use topology to define an equivalence relation
on points $(a,b)$,
in the $ab$-plank, for which the corresponding partitionings of their fundamental regions ``look essentially the same''. We do not technically have a partition since
the subsets on which each distance functions are minimal can overlap
on boundary curves.

Let $d_\alpha$ denote an element of the set $\{ d_{R,0}, d_{R,1},
d_{R,2}, d_{U,0}, d_{U,1}, d_{U,2}\}$. For each fundamental region
$\calF$ we call a connected component of the subset on which
$d_\alpha$ is smallest a $d_\alpha$-region. For a given
fundamental region $\calF$, let $S$ be the union of the borders of
the $d_\alpha$-regions (including the four sides of $\calF$). We
use $\prod_0 (\calF\setminus S)$ to denote the set of connected
components of $\calF\setminus S$; each is a $d_\alpha$-region. Let
$f:\prod_0 (\calF\setminus S)\rightarrow \{ d_{R,0}, d_{R,1},
d_{R,2}, d_{U,0}, d_{U,1}, d_{U,2}\}$ be the function that sends a
$d_\alpha$-region to $d_\alpha$.
Let $Q$ and $R$ be points in the $ab$-plank; we use $Q$ and $R$ as subscripts
to indicate to which point a particular notation is associated.
We say that
$Q$ is equivalent to $R$ if and only if there is a
homeomorphism $\iota : \calF_Q\ra \calF_R$ with the following
properties: i) $\iota|_{S_Q}$ induces a homeomorphism of $S_Q$ and
$S_R$,  ii) $\iota$ sends $(0,0)$, $(0,\frac{1}{2})$,
$(\frac{a_Q}{2},0)$ and $(\frac{a_Q}{2},\frac{1}{2})$ to $(0,0)$,
$(0,\frac{1}{2})$, $(\frac{a_R}{2},0)$ and
$(\frac{a_R}{2},\frac{1}{2})$, respectively, and iii) we have
$f_R\circ \iota|_{\prod_0(\calF_Q\setminus S_Q)}=f_Q$ (where
$\iota|_{\prod_0(\calF_Q\setminus S_Q)}$ is the obvious induced
map). Our goal is to find the equivalence classes for this
relation. We call a pair $(\calF, f)$  a {\it labeled} $\calF$. For each
equivalence class, we will also describe the associated labeled
$\calF$'s.

\section{The equivalence classes}
\label{TheEC}

\vspace{3mm}\noindent Now we want to partition the $ab$-plank by
the equivalence classes defined in Section~\ref{equrel}.
We will prove that there are
47 of them, some having area, some having length and two are
single points. We use reverse lexicographical order on $(a,b)$ to
order the equivalence classes. In Table~\ref{47ec} we define the
47 equivalence classes, give the dimension of each and give a set
of equalities and inequalities that define the subset of the
$ab$-plank that is the given equivalence class. In
Table~\ref{47ec}, $a'$ is the root of $a^4-2a^3+7a^2-6a+1$ near
$0.780$, $a''$ is the root of $8a^5+12a^4+a^3-10a^2-6a-1$ near
$0.927$, $b'$ is the root of $4b^3+3b-6$ near
$0.929$, $b''$ is the root of
$6b^7-7b^6-12b^5+7b^4+8b^3-b^2-2b-1$ near 1.72, and $b'''$ is the root of
$3b^4-10b^3+11b^2-6b +1$ near $1.92$.

All of the curves listed in Table~\ref{47ec} are lines or conic
sections, and so are easy to graph, except the two quartics
$(b^2+1)a^2+2b^3a-2b^3-b^2 = 0$ and $(2b-2)a^3+(b^2-1)a^2-b^2 =
0$. In addition, for the proof of Proposition~\ref{U2smallest}, we
will need to graph $(b^2+1)a^2+2a-b^2-2b=0$. The software Magma
(see \cite{Ma}) shows that each is a singular curve of genus 1 and
finds a real birational map to the projective closures of the
non-singular curves $y^2=x^3+19x^2+120x+256$, $y^2=x^3+4x^2+256$
and $y^2=x^3+x^2+4$, respectively (see \cite{Ha}). Since each
cubic in $x$ has a single real root, each of those projective
curves has a single real component. The images of each of these
single real components in $ab$-space go off to infinity; so
none of the quartics have a compact component. So we can trust
graphing software to draw them without missing a small, compact
component.

\begin{table}

 {\scriptsize
\begin{tabular}{l|l|l}
Equiv & & \\
class & dim & definition \\ \hline
1 & 2 &  $a+b-1 \leq 0$, $a^2+2ab-2b\leq 0$ \\
2 & 2  & $a+b-1 \leq 0$, $a^2+2ab-2b > 0$ \\
3 & 2 &  $a+b-1
> 0$, $a^2+2ab-2b > 0$, $ab-2a+2b^2-3b+2>0$, $b < \frac{2}{3}$ \\
4 & 1 &  $ab-2a+2b^2-3b+2=0$,  $ b <
\frac{2}{3}$ \\
5 & 2 &  [$ab-2a+2b^2-3b+2< 0$, $b\leq \frac{2}{3}$, $a < 1$] $\cup$ \\

& &  [$ (b^2+1)a^2+2b^3a-2b^3-b^2 > 0$,
$a-b^2 \geq 0$, $\frac{2}{3} < b < 1$, $a < 1$] \\
6 & 1 &  $a=1$, $0 < b \leq 1$ \\
7 & 2 &  $a+b-1 > 0$, $a^2+2ab-2b\leq 0$, $2b+a-2 < 0$ \\
8 & 1 & $2b+a-2=0$, $0 < a \leq \frac{2}{3}$ \\
9 & 2 &  $2b+a-2 > 0$, $a^2+(2b+2)a-4b \leq 0$,
$a^2+(2b-2)a+3b^2-2b-1 \leq 0$ \\
10 & 2 & $a^2+(2b+2)a-4b > 0$, $(b^2+1)a^2+2b^3a-2b^3-b^2 < 0$,
$a^2+(2b-2)a+3b^2-2b-1 \leq 0$ \\
11 & 1 & $(b^2+1)a^2+2b^3a-2b^3-b^2 = 0$, $\frac{2}{3} < a \leq
a'$\\
12 & 2 & $a^2+(2b+2)a-4b
> 0$, $(b^2+1)a^2+2b^3a-2b^3-b^2 \leq 0$, \\ & & $a^2+(2b-2)a+3b^2-2b-1
>0$\\
 13 & 2 & $(b^2+1)a^2+2b^3a-2b^3-b^2 >
0$, $a^2+(2b+2)a-4b > 0$,  $(\sqrt{2}+1)a-b-1 < 0$, $a' < a < 2\sqrt{2}-2$\\
14 & 1 & $(\sqrt{2}+1)a-b-1 = 0$, $a' < a < 2\sqrt{2}-2$ \\
15 & 2 &  $a-b^2 < 0$, $(\sqrt{2}+1)a-b-1 > 0$, $b < 1$\\
16 & 2 & $a^2+(2b+2)a-4b \leq 0$, $(b^2+1)a^2+2b^3a-2b^3-b^2 \leq
0$, \\ & & $a^2+(2b-2)a+3b^2-2b-1 > 0$, $b < 1$ \\ 17 & 2 &
$(b^2+1)a^2+2b^3a-2b^3-b^2
> 0$, $b' < b < 1$, $a^2+(2b+2)a-4b \leq 0$ \\ 18  & 1 & $b=1$, $0 < a \leq
\frac{-1+\sqrt{7}}{2}$ \\
19 & 1 & $b=1$, $\frac{-1+\sqrt{7}}{2} < a < 2\sqrt{2}-2$ \\ 20 &
0 &  $(a,b)=(2\sqrt{2}-2,1)$ \\ 21 & 1 & $b=1$, $2\sqrt{2}-2 < a <
1$ \\ 22 & 2 &  $a-b+1 \leq 0$, $2ab-2a-1 < 0$ \\
23 & 2 & $a-b+1 > 0$, $2a^2+(-3b+1)a+2b^2-2b>0$, \\ & &
$2ab-2a-1<0$, $(b^2+1)a^2+2b^3a-2b^3-b^2 \leq 0$, $b>1$ \\ 24 & 2
& $2a^2+(-3b+1)a+2b^2-2b \leq 0$, $a-2b+2 < 0$,
$(b^2+1)a^2+2b^3a-2b^3-b^2 \leq 0$ \\ 25 & 2 & $b > 1$, $a-2b+2
\geq 0$, $(b^2+1)a^2+2b^3a-2b^3-b^2 \leq 0$ \\  26 & 2 & $b
> 1$, $(b^2+1)a^2+2b^3a-2b^3-b^2 > 0$, $a^2+(2b+2)a-4b < 0$,
$a-2b+2 \geq 0$ \\  27 & 1 & $a^2+(2b+2)a-4b = 0$,  $2\sqrt{2}-2 <
a < 1$ \\ 28 & 2 & $a < 1$, $1 < b$, $a^2+(2b+2)a-4b
> 0$ \\ 29 & 1 &  $a=1$, $1 < b < \frac{3}{2}$
\\ 30 & 2 & $a-2b+2< 0$,
$(b^2+1)a^2+2b^3a-2b^3-b^2 > 0$, $2a^2+(-3b+1)a+2b^2-2b \leq 0$ \\
31 & 2 & $2a^2+(-3b+1)a+2b^2-2b>0$, $(b^2+1)a^2+2b^3a-2b^3-b^2 >
0$, $2ab-2a-1<0$, $b > 1$ \\
32 & 0 & $(a,b)=(1,\frac{3}{2})$ \\
33 & 1 & $2ab-2a-1=0$, $a'' < a < 1$ \\
34
 & 2 &
$2ab-2a-1> 0$, $(b^2+1)a^2+2b^3a-2b^3-b^2
> 0$, $(2b-2)a^3+(b^2-1)a^2-b^2 < 0$ \\  35 & 1 &
$(2b-2)a^3+(b^2-1)a^2-b^2 = 0$, $\frac{3}{2} < b < b''$ \\ 36 & 2
& $(2b-2)a^3+(b^2-1)a^2-b^2 > 0$, $(b^2+1)a^2+2b^3a-2b^3-b^2 > 0$,
$a-b+1 > 0$, $a < 1$ \\  37 & 1 & $a=1$, $\frac{3}{2} < b < 2$ \\ 38 & 1 &
$2ab-2a-1=0$,
$\frac{1}{\sqrt{2}} < a\leq a''$ \\
39 & 2 & $2ab-2a-1 > 0$, $(b^2+1)a^2+2b^3a-2b^3-b^2  \leq 0$, \\
& & $(2b-2)a^3+(b^2-1)a^2-b^2 < 0$, $a-b+1 > 0$ \\
40 & 1 &  $2ab-2a-1=0$, $1+\frac{1}{\sqrt{2}}\leq b$ \\
41 & 2 &  $2ab-2a-1 > 0$, $a-b+1 \leq 0$,
$(2b-2)a^3+(b^2-1)a^2-b^2
< 0$ \\
42 & 1 &   $(2b-2)a^3+(b^2-1)a^2-b^2 = 0$,   $ b'' \leq b < b'''$ \\
43 & 2 & $(2b-2)a^3+(b^2-1)a^2-b^2 > 0$,
$(b^2+1)a^2+2b^3a-2b^3-b^2  \leq 0$, $a-b+1 > 0$ \\
44 & 1 & $(2b-2)a^3+(b^2-1)a^2-b^2 = 0$, $b''' \leq b$  \\
45 & 2 & $(2b-2)a^3+(b^2-1)a^2-b^2  > 0$,
$(b^2+1)a^2+2b^3a-2b^3-b^2
\leq 0$,  $a-b+1 \leq 0$\\
46 & 2 & $(b^2+1)a^2+2b^3a-2b^3-b^2 >  0$, $a-b+1 \leq 0$, $a < 1$  \\
47 & 1 & $a= 1$, $2 \leq b$ \\
\end{tabular}  }
\caption{The 47 equivalence classes}
\label{47ec}
\end{table}

In Figure~\ref{partition}, we show the partition of the
$ab$-plank into the 47 equivalence classes.
For 2-dimensional equivalence classes, we use arrows to denote which boundaries
are part of the equivalence class.
It is difficult to
include equivalence classes 10 - 14, 17, 19, 20 and 31 - 33 in our figure
as they are small.
For example,  four of the 2-dimensional classes, 10, 12, 13 and 17, have
areas that are approximately 0.00075, 0.000062, 0.000035 and
0.00018, respectively.

\begin{figure}
\centering
\includegraphics[width=7cm]{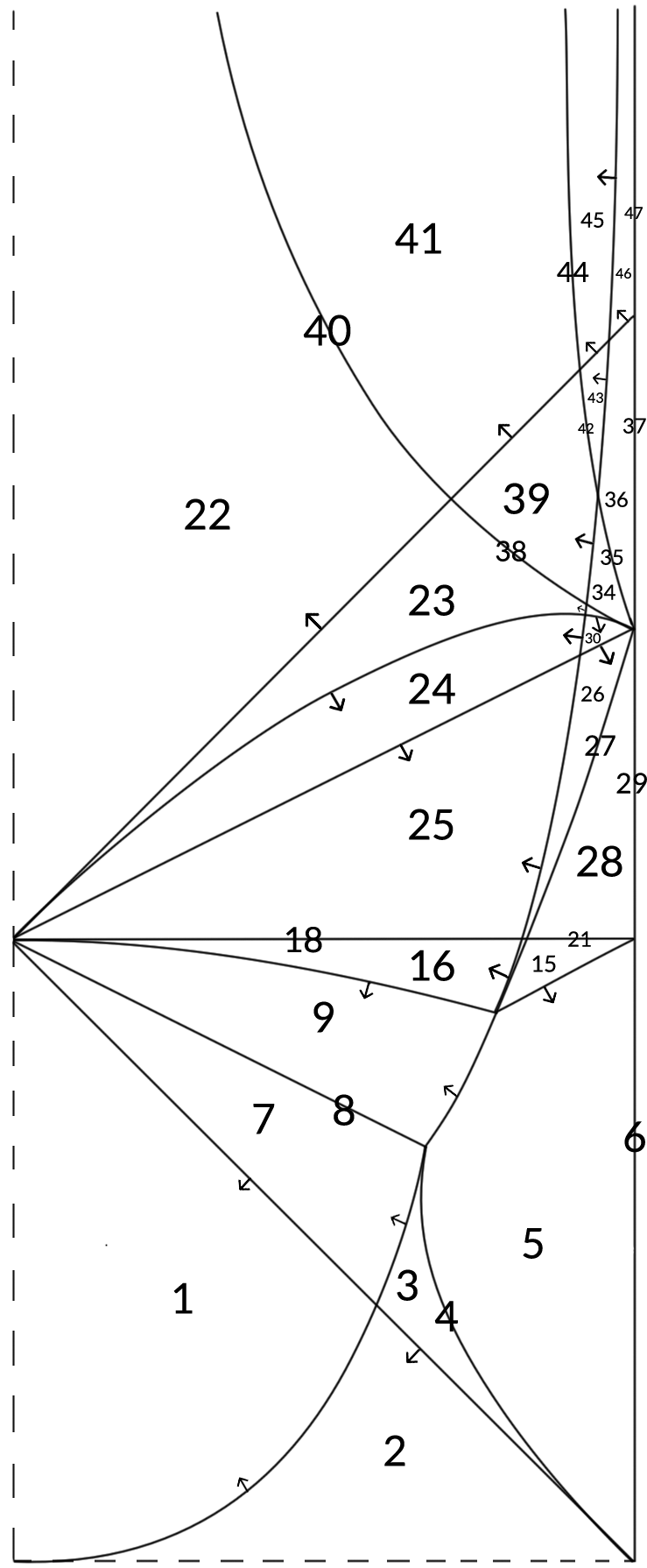}
\caption{The 47 equivalence classes in the $ab$-plank}
\label{partition}
\end{figure}

In Figures~\ref{TheFs1} and \ref{TheFs2}, we present a labeled $\calF$
for each of the 47 equivalence classes, that is
homeomorphic (with the three properties listed
at the end of Section~\ref{equrel}) to the labeled $\calF$ for
each $(a,b)$ in the equivalence class.
For simplicity, we give only the subscripts of the labeling and omit the $d$'s
and commas.

A priori, it seems surprising in equivalence class 13 that there are disjoint subsets of $\calF$ on
which $d_{U,0}$ is smallest.
We will see, in Section~\ref{DistPairs},  for 13 of the pairs $d_\alpha, d_\beta$, that
$d_\alpha=d_\beta$ is a hyperbola in the $xy$-plane. So, without loss of generality, there are two disjoint components of the $xy$-plane on which $d_\alpha < d_\beta$. Perhaps we should be surprised that there is a unique
equivalence class in which this occurs in $\calF$.

We also note, when passing from equivalence class 5 to 15,
a subset arises in the interior of $\calF$ in which $d_{R,1}$ is smallest.
Such subsets also arise in the interior of the $1\times a$ face, though on a border of $\calF$:
when crossing
$(b^2+1)a^2+2b^3a-2b^3-b^2 = 0$ for $a > a'$, from left to right,
and when crossing
$a^2+(2b-2)a+3b^2-2b-1 = 0$ for $a < a'$, from below to above, interior
subsets arise in which $d_{U,0}$ and $d_{R,1}$, respectively, are smallest.

\begin{figure}
\centering
\includegraphics[width=9cm]{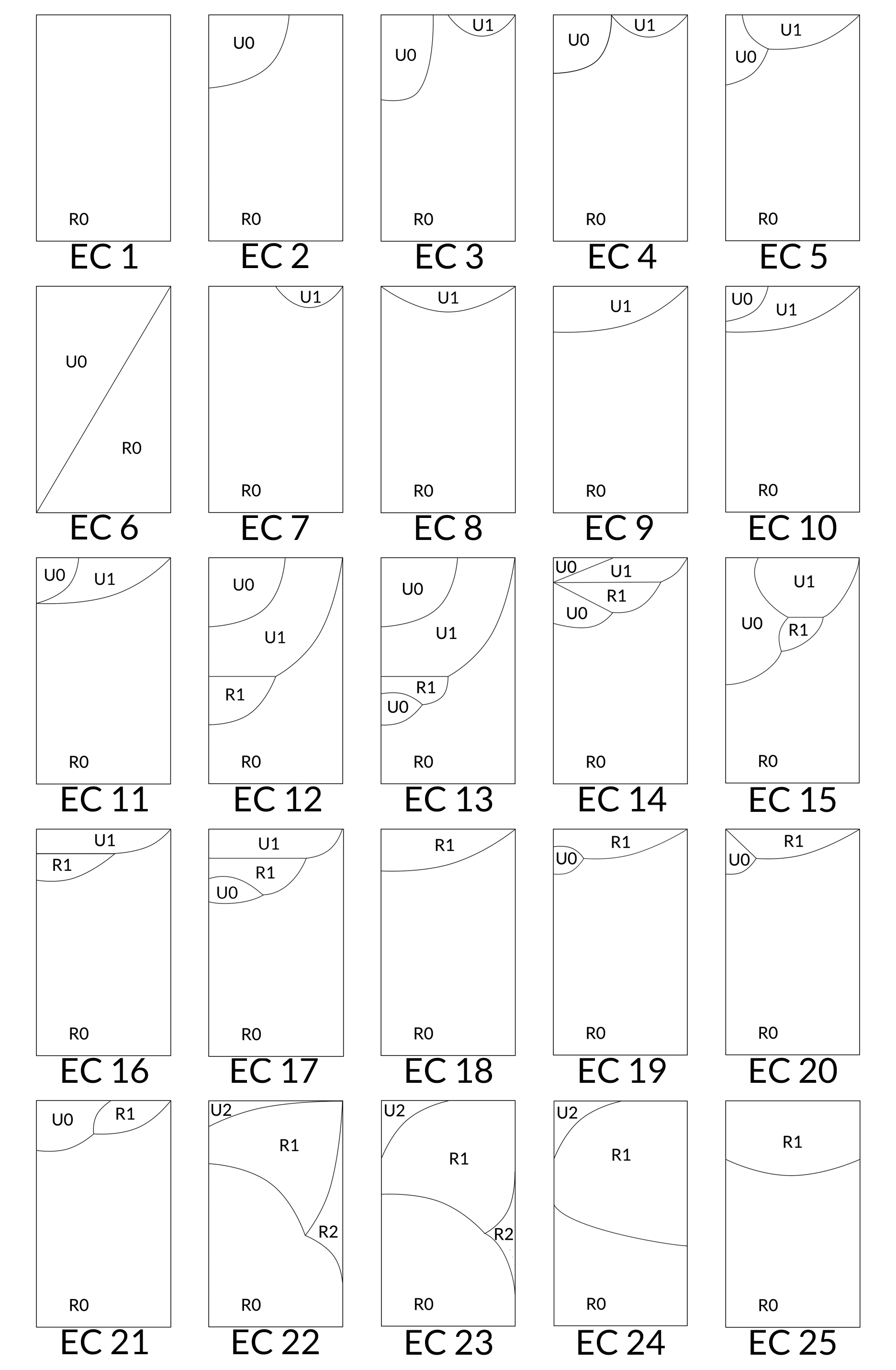}
\caption{Labeled $\calF$'s for Equivalence Classes 1 - 25}
\label{TheFs1}
\end{figure}

\begin{figure}
\centering
\includegraphics[width=9cm]{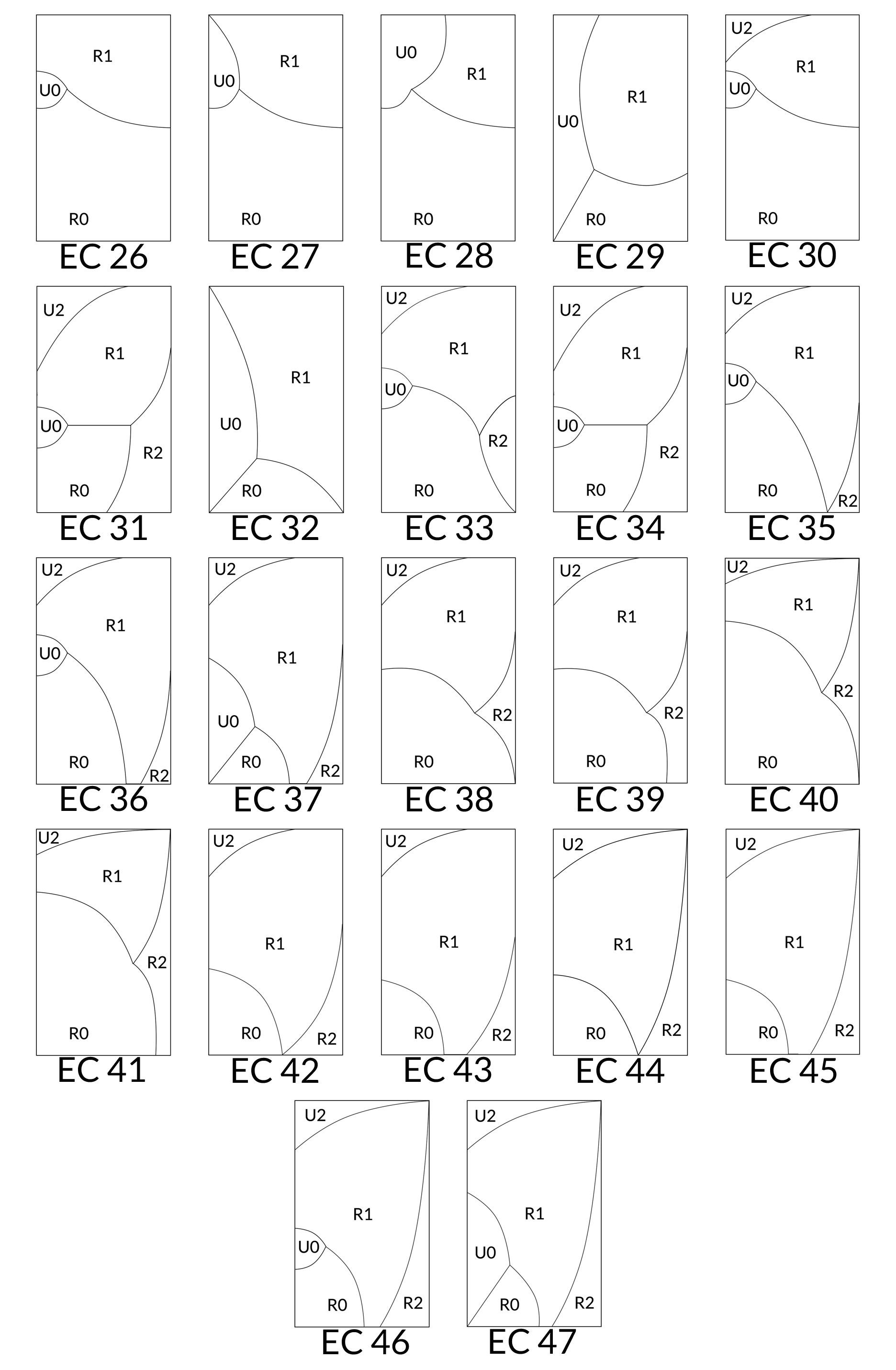}
\caption{Labeled $\calF$'s for Equivalence Classes 26 - 47}
\label{TheFs2}
\end{figure}

\section{Distance functions in pairs}
\label{DistPairs}

It is useful for understanding where the equivalence classes change in the $ab$-plank to equate the six distance functions, two at a time.

\begin{lemma}\label{DisDifIsoClaCha}
Fix two distinct distance functions $d_\alpha$, $d_\beta$, from the six of concern. Let
${\mathcal P}$ be a  path in the $ab$-plank. For each
point in the path there is a labeled $\calF$, considering only
$d_\alpha$ and $d_\beta$. If the equivalence classes for these
labeled $\calF$'s are not all the same for the points on
${\mathcal P}$ then one of the following occurs at some point on
${\mathcal P}$:
 1)
$d_\alpha=d_\beta$ at a corner of $\calF$, 2) $d_\alpha=d_\beta$
has a double intersection with a side of $\calF$, 3)
$d_\alpha=d_\beta$ consists of two line
segments meeting in the interior of $\calF$.
\end{lemma}

Note that a double intersection with a side of $\calF$ is either a
tangency, or the  two lines of a degenerate hyperbola
crossing a side, each transversally, at the same point.  In iii),
$d_\alpha=d_\beta$ is part of a degenerate hyperbola.

\begin{proof}
From Table~\ref{10distances},
we have $d_{R,0}=d_{U,2}$ and $d_{R,2}=d_{U,0}$ given by
$(2b+2)y+ab-a-b-1=0$ and $(2a+2b)x-a^2-ab-a+b=0$, respectively.
If we fix $(a,b)$ in the $ab$-plank, then the extension of each of these equations to
the entire $xy$-plane describes a line.
The other ${6\choose 2}-2=13$ equations, obtained by
setting two distinct distance functions to be equal,
are of the
form $\gamma x^2+\delta xy + \phi y^2 + \ell_1(a,b)x+\ell_2(a,b)y
+q(a,b)=0$, where $\gamma, \delta, \phi$ are constants (i.e.\ not depending
on $a$ or $b$), the $\ell_i$ are polynomials of degree at most 1
and $q$ is a polynomial of degree at most 2. In all 13 cases we have
$\delta^2 - 4\gamma\phi > 0$. So if we fix $a$ and $b$, the resulting equation,
when extended to the entire $xy$-plane, describes a hyperbola (possibly degenerate), which
we sometimes refer to for clarity.

We fix a pair $d_\alpha, d_\beta$.
Assume we do not have equivalence of all points of ${\mathcal P}$.
Since each $d_\alpha-d_\beta$ is a polynomial,
a continuous parametrization of ${\mathcal P}$ induces
a continuously parameterized family
 of lines or hyperbolas. Thus
there is a point $Q\in {\mathcal P}$ with the property that for every
$\epsilon > 0$, sufficiently small, and for all points $R\neq Q$ on
a fixed side of $Q$ in ${\mathcal P}$,
with ${\rm dist}(Q,R) < \epsilon$, we have  $Q$ and $R$
not equivalent.

Assume, for every $\epsilon > 0$, sufficiently small, and for all points $R\neq Q$ on a fixed side of $Q$ in ${\mathcal P}$, with ${\rm dist}(Q,R) < \epsilon$, that there is  a homeomorphism $\iota$ from $\calF_Q$ to $\calF_{R}$ satisfying
conditions i) and ii) in the definition
of the equivalence relation in Section~\ref{equrel}.
Let $\bar{\iota}$ denote the map induced from
the set of open, connected components of $\calF_Q\setminus S_Q$
to that of $\calF_{R}\setminus S_{R}$ (recall $S$ is the union of $d_\alpha=d_\beta$ and the border of $\calF$).
Recall that we have a continuously parameterized family of lines or hyperbolas. Thus for all open, connected components $U$ of $\calF_Q\setminus S_Q$, there
is an $(x,y)\in U$ such that
for all $\epsilon$, sufficiently small,
we have $(x,y)\in \bar{\iota}(U)$ as well.
We fix a continuous parametrization of ${\mathcal P}$ by the variable $t$.
From above, there is no $t$ between
$t_Q$ and $t_{R}$ such that $(d_\alpha-d_\beta)(a(t),b(t),x,y)=0$.
So by the Intermediate Value Theorem, $d_\alpha-d_\beta$ takes the same
sign at $(a(t_Q),b(t_Q),x,y)$ and $(a(t_{R}),b(t_{R}),x,y)$.
Thus condition iii) holds for $\iota$ as well.

We now prove the contrapositive of what remains to be proven.
Assume that  for $Q$, none of 1), 2) or 3) from this Lemma occur.
There are 13 classes for  $\calF_Q$ and $S_Q$, up to rotation, reflection, and homeomorphism, preserving conditions i) and ii) of the
equivalence relation; representatives are given in Figure~\ref{13ECs}.
Informally, we see that small perturbations in each of the 13 representatives in
Figure~\ref{13ECs} do not lead to a change in equivalence class.
In other words, we see
 for all $\epsilon > 0$,  sufficiently small, and all points
$R\neq Q$ on a fixed side of $Q$ in ${\mathcal P}$, with ${\rm dist}(Q,R)< \epsilon$,
that
there is a homoeomorphism from $\calF_Q$ to
 $\calF_{R}$
satisfying
conditions i) and ii) of the equivalence relation.
\end{proof}

\begin{figure}
\centering
\includegraphics[width=9cm]{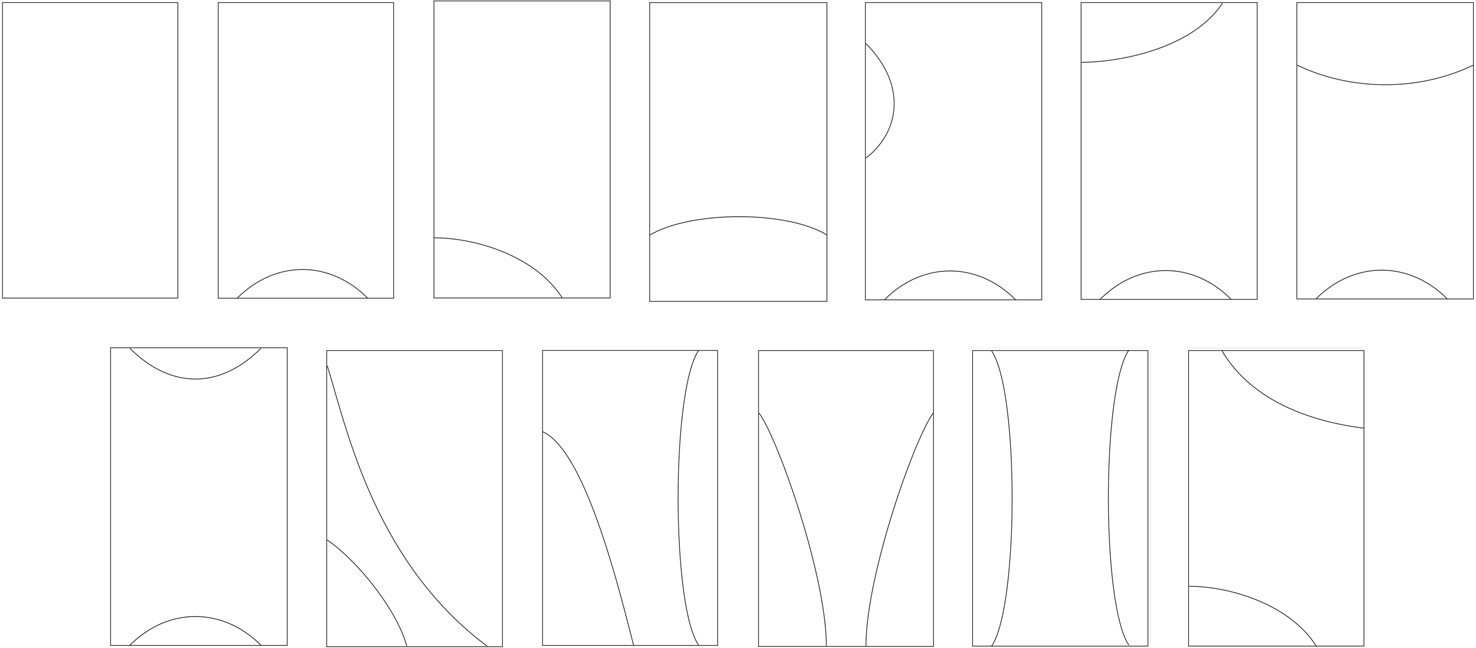}
\caption{13 Equivalence Classes for Lemma~\ref{DisDifIsoClaCha}}
\label{13ECs}
\end{figure}

For each pair $d_\alpha, d_\beta$ we determine where, in the
$(a,b)$ plank, each of the three conditions in
Lemma~\ref{DisDifIsoClaCha} can occur (taking each corner and side
of $\calF$ into account). Most determine a curve in the
$ab$-plank.  We then
break the $ab$-plank into equivalence classes for the labeled $\calF$'s taking only $d_\alpha$ and $d_\beta$ into
account.
For nine of these pairs, this information is useful for determining
the 47 equivalence classes and this is presented in Section~\ref{DistFcnPair}.
We do not attempt such computations for subsets of three or more of the distance functions simultaneously, because
the possibilities lead to a combinatorial explosion.

\section{Eliminating distance functions}
\label{elim}

For each of the six distance functions, we want to determine
subsets of the $ab$-plank for which there is no $(x,y)$, in  the
associated fundamental region $\calF$, for which the given distance function is smallest.
We use the results from Section~\ref{DistFcnPair} to prove  in
Section~\ref{elim2} that certain
distance functions can not be smallest on certain open subsets of
the $ab$-plank. The interior of each 2-dimensional equivalence
class is contained in one of these open subsets. As for 1- and
0-dimensional equivalence classes, by continuity, only those
distance functions that are smallest in all of the bordering
subsets can be smallest. In Table~\ref{Naive}, we record this
information by equivalence class.

\begin{table}

\begin{tabular}{l|l}
Equivalence classes & functions that can be smallest \\ \hline 1 & $d_{R,0}$ \\
2, 6 & $d_{R,0}$, $d_{U,0}$ \\
3 - 5, 10, 11 & $d_{R,0}$, $d_{U,0}$, $d_{U,1}$ \\
7 - 9 & $d_{R,0}$, $d_{U,1}$ \\
12 - 15, 17 & $d_{R,0}$, $d_{R,1}$, $d_{U,0}$, $d_{U,1}$ \\
16 & $d_{R,0}$, $d_{R,1}$,  $d_{U,1}$ \\
18, 25 & $d_{R,0}$, $d_{R,1}$  \\
19 - 21, 26 - 29, 32 & $d_{R,0}$, $d_{R,1}$,  $d_{U,0}$ \\
22, 23, 38 - 45 & $d_{R,0}$, $d_{R,1}$,  $d_{R,2}$, $d_{U,2}$ \\
24 & $d_{R,0}$, $d_{R,1}$,  $d_{U,2}$ \\
30 & $d_{R,0}$, $d_{R,1}$, $d_{U,0}$,  $d_{U,2}$ \\
31, 33 - 37, 46, 47 & $d_{R,0}$, $d_{R,1}$, $d_{R,2}$, $d_{U,0}$,  $d_{U,2}$ \\
\end{tabular}
\caption{Where distance functions can be smallest}
\label{Naive}
\end{table}

\section{Spontaneous generation of triple intersections}

In this section, we describe the most interesting proof technique
we used (in Sections~\ref{blessthan1} and \ref{bbiggerthan1}) to determine the equivalence classes
and associated labeled $\calF$'s.
Let $d_\alpha$, $d_\beta$, $d_\gamma$ be three distinct distance functions
from the six of concern. If, for a given
$(a,b)$, we have $(x,y)\in \calF$ such that $d_\alpha
(x,y)=d_\beta (x,y)=d_\gamma (x,y)$ then we call $(x,y)$ a {\it
triple intersection} for $d_\alpha, d_\beta, d_\gamma$.

It will sometimes help to know, for a given three distance functions, that there is no
triple intersection in the interior of $\calF$.
Let us start with a geometric example of what we call a
spontaneous generation of a triple intersection. Then we give
a rigorous definition and a description of how to find them. As an
example, let $d_\alpha, d_\beta, d_\gamma$ denote three distinct
distance functions. Assume for all $(a,b)$ in some open subset $U$ of the $ab$-plank,
that $d_\alpha=d_\beta$  is a single arc, concave up, with lowest point
$(\frac{a}{4},y_L)$ (with $0 < y_L < \frac{1}{2}$ and $y_L$ depending on $(a,b)$); $d_\alpha$ is smaller
above the arc and $d_\beta$ below.
 Assume  for all $(a,b)$ in $U$,   that $d_\beta=d_\gamma$ is a single
arc, concave down, with highest point $(\frac{a}{4},y_U)$ (with
$0 < y_U < \frac{1}{2}$ and $y_U$ depending on $(a,b)$); $d_\beta$ is smaller above the arc and $d_\gamma$
below. Assume that there is a path in $U$ from $(a_0,b_0)$ to $(a_1,b_1)$
(distinct points)
 and on all points of the path except $(a_0,b_0)$
we have $y_U < y_L$ and at $(a_0,b_0)$ we have $y_U=y_L$. Then at
$(a_0,b_0)$ we have the spontaneous generation of a triple
intersection for $d_\alpha$, $d_\beta$, $d_\gamma$ at
$(\frac{a}{4},y_L)$. Note that if we consider only
$d_\alpha$, $d_\beta$ and $d_\gamma$, then the labeled $\calF$
at $(a_0,b_0)$ is not equivalent to those for other points on the path.

Now we give a rigorous definition. Let $d_\alpha, d_\beta,
d_\gamma$ denote three distinct distance functions from the six of
concern. Assume there is a given $(a_0,b_0)$ for which
$d_\alpha=d_\beta=d_\gamma$ at some point $(x_0,y_0)$ in the
interior of $\calF$. Assume also that there is no neighborhood of
$(a_0,b_0)$  inside which, for every point $(a,b)$ in the
neighborhood we have $d_\alpha=d_\beta=d_\gamma$ at some point
$(x,y)$ in the interior of $\calF$. Then we define
$(a_0,b_0,x_0,y_0)$ to be a {\it spontaneous generation} of a
triple intersection for $d_\alpha, d_\beta, d_\gamma$.

To find the set of $(a,b)$'s at which they occur,
we can take the resultant of $d_\alpha-d_\beta$ and
$d_\alpha-d_\gamma$ with respect to $y$. This defines a surface
(with possibly more than one component) in $abx$-space. We then
take the projection of this surface onto the $ab$-plank. If
$(a_0,b_0,x_0,y_0)$ is a spontaneous generation, then the
projection of $(a_0,b_0,x_0)$ on the $ab$-plank will be on the
boundary of the projection of a neighborhood of $(a_0,b_0,x_0)$
in the surface.
At such a point $(a_0,b_0,x_0)$, a normal vector to the
surface will be parallel to the $ab$-plank. So we can find
conditions on $a$ and $b$ for
such
points by taking the resultant, with respect to $x$, of a
polynomial defining the surface, with its partial derivative with
respect to $x$. The following theorem will be used in two of the
proofs in Section~\ref{bbiggerthan1}
determining the equivalence classes.

\begin{theorem}
\label{spontaneous} For the  triples $d_{R,1}, d_{R,2}, d_{U,2}$ and
$d_{R,1}, d_{R,2}, d_{U,0}$
 there is
no spontaneous generation of a triple intersection.
\end{theorem}

\begin{proof}
For $d_{R,1}, d_{R,2}, d_{U,2}$ the points obtained at which a
spontaneous generation of a
triple intersection could occur are on $b(a+1)(2a+2+b)=0$, which
does not pass through the $ab$-plank. For
$d_{R,1}, d_{R,2}, d_{U,0}$, after using a resultant to remove $y$
we get $4(a^2-2bx+ab+a-b-2ax)^2$. We then take the resultant, with respect to $x$, of
$a^2-2bx+ab+a-b-2ax$ with its partial derivative with respect to
$x$ to get $a+b$. But $a+b=0$ does not pass through the
$ab$-plank.
\end{proof}

\section{The equivalence class computations}

For each subset of the $ab$-plank described in Table~\ref{47ec},
we need only consider the distance functions listed in Table~\ref{Naive}.
Then, in the proofs in Sections~\ref{blessthan1} and \ref{bbiggerthan1}, we show that changes
in equivalence class can only occur
along certain curves in the $ab$-plank. Nine of these equivalence class borders come from the $ab$-curves
associated to the occurrences
described in
Lemma~\ref{DisDifIsoClaCha}. Five of these borders are where three distance functions have
a triple intersection on one of the four sides of $\calF$. The curve
$a-b^2=0$,
for $a' \leq a \leq 1$, is where $d_{R,0}=d_{R,1}=d_{U,0}=d_{U,1}$ at a point in $\calF$
and the curve
$(\sqrt{2}+1)a-b-1=0$, for $a' \leq a \leq 2\sqrt{2}-2$,
is where the hyperbolas $d_{R,1}=d_{U,0}$ and $d_{U,0}=d_{U,1}$ are both degenerate and the four lines making up those two hyperbolas all meet at one point on $x=0$ in $\calF$.

In
 Sections~\ref{blessthan1} and \ref{bbiggerthan1}, we also
show that the labeled $\calF$, for each $(a,b)$ in the given
subset, is homeomorphic (with the three properties described in the definition of the equivalence relation
in Section~\ref{equrel} of this article) to the one in Figure~\ref{TheFs1} or \ref{TheFs2}.
As there is no such homeomorphism between any distinct pair of
labeled $\calF$'s in Figure~\ref{TheFs1} and \ref{TheFs2}, we then know that each
of the 47 subsets
is a distinct equivalence class.

\section{Conclusion}

In Dudeney's problem we have a $12\times 12\times 30$ foot room.
The spider is $1$ foot below the ceiling and half way between the
sides of the room. In our notation we have
$(a,b,x,y)=(1,\frac{5}{2},0,\frac{5}{12})$. This $(a,b)$ is in
equivalence class 47 and this $(x,y)$ is on the border of the
region where $d_{U,2}$ ($=d_{U,-2}$) is smallest. Indeed, on the
path $p_{U,2}$,  the spider must cross 5 sides for the shortest
path to the fly. If, instead, the spider is $3$ feet below the
ceiling of the same room then we have $(x,y)=(0,\frac{3}{12})$.
This $(x,y)$ is on the border of where $d_{R,1}$ ($=d_{U,-1}$) is
smallest and the shortest path to the fly opposite crosses 4 sides of the room. Lastly, if the spider is
$5$ feet below the ceiling of the same room then we have
$(x,y)=(0,\frac{1}{12})$. This $(x,y)$ is on the border of where
$d_{U,0}$ is smallest and the obvious path, straight up, straight
across the top and then straight down to the fly opposite is the
shortest.

\section*{About the authors:}
  S.~ Michael Miller was an undergraduate at the time of the writing of this article
and is now pursuing his Ph.D.\ in Mathematics at the University of California, Los Angeles.
Edward Schaefer is a professor, whose area is arithmetic geometry. He taught
cryptography for two years at Mzuzu University in Malawi recently.

\subsection*{S.~ Michael Miller}~
UCLA Mathematics Department,
Box 951555
Los Angeles, CA 90095-1555.
smmiller@g.ucla.edu

\subsection*{Edward F.~Schaefer}~
Department of Mathematics and
Computer Science,
Santa Clara University,
Santa Clara, CA 95053.
eschaefer@scu.edu

\vspace{3mm}
\noindent
\begin{center}
{\large\bf APPENDIX}
\end{center}

\section{Distance functions in pairs}
\label{DistFcnPair}

\begin{lemma}
\label{R0R1} Below $a+2b-2=0$, $d_{R,0} < d_{R,1}$. Between
$a+2b-2=0$ and $b=1$, $d_{R,0}=d_{R,1}$ is a single arc, concave
up, with positive slopes, meeting $x=0$ with $0 < y < \frac{1}{2}$
and $y=\frac{1}{2}$ with $0 < x < \frac{a}{2}$. To the left of the
arc, $d_{R,1}$ is smaller and to the right, $d_{R,0}$ is.

On $b=1$, $d_{R,0}=d_{R,1}$ is a single arc with right endpoint
$(\frac{a}{2},\frac{1}{2})$. The left endpoint is on $x=0$ with $0
< y < \frac{1}{2}$. Above the arc, $d_{R,1}$ is smaller and below
it, $d_{R,0}$ is.

For $b > 1$, $d_{R,0}=d_{R,1}$ is a single arc. Its left endpoint
is on $x=0$ with $0 < y < \frac{1}{2}$. To the lower left of
$(4b-2)a-2b-1=0$, the right endpoint is on $x=\frac{a}{2}$ with $0
< y < \frac{1}{2}$. On $(4b-2)a-2b-1=0$, the right endpoint is
$(\frac{a}{2},0)$. Above $(4b-2)a-2b-1=0$,  the right endpoint  is
on $y=0$ with $0 < x < \frac{a}{2}$. Above the arc, $d_{R,1}$ is
smaller and below it, $d_{R,0}$ is.
\end{lemma}

Note $a+2b-2=0$ is where $d_{R,0}=d_{R,1}$ at $(0,\frac{1}{2})$.

\begin{lemma}
\label{R0R2} For $b < 1$, $d_{R,0} < d_{R,2}$. For $b > 1$,
$d_{R,0}=d_{R,2}$ is a single arc with negative slopes (for $0 < y$). Its upper
endpoint is on $y=\frac{1}{2}$ with $0 < x < \frac{a}{2}$. Below
$2ab-2a-1=0$, the lower endpoint is on $x=\frac{a}{2}$ with $0 < y
< \frac{1}{2}$. On $2ab-2a-1=0$, the lower endpoint is
$(\frac{a}{2},0)$. Above $2ab-2a-1=0$ , the lower endpoint is on
$y=0$ with $0 < x < \frac{a}{2}$. To the left of the arc,
$d_{R,0}$ is smaller and to the right of it, $d_{R,2}$
is.\end{lemma}

Note $b=1$ is where $d_{R,0}=d_{R,2}$ at $(\frac{a}{2},\frac{1}{2})$.

\begin{lemma}
\label{R0U0} On $a=1$, $d_{R,0}=d_{U,0}$ is a line segment
connecting $(0,0)$ and $(\frac{1}{2},\frac{1}{2})$. Between $a=1$
and $a^2+2ab-2b=0$, $d_{R,0}=d_{U,0}$ is a single arc, concave up,
with non-negative slopes, with left endpoint on $x=0$ with $0 < y
< \frac{1}{2}$ and right endpoint on $y=\frac{1}{2}$ with $0 < x <
\frac{a}{2}$. Above the arc, $d_{U,0}$ is smaller and below it,
$d_{R,0}$ is. Above $a^2+2ab-2b=0$,  $d_{R,0} < d_{U,0}$.
\end{lemma}

Note $a^2+2ab-2b=0$ is where $d_{R,0}=d_{U,0}$ at
$(0,\frac{1}{2})$.

From Lemma~\ref{R1U1}, it suffices to consider $d_{U,1}$ for $b < 1$.

\begin{lemma}
\label{R0U1} For all $(a,b)$, $d_{R,0}=d_{U,1}$ at $(\frac{a}{2},\frac{1}{2})$.
Below $a+b-1=0$, $d_{R,0}=d_{U,1}$ meets $\calF$ nowhere else  and $d_{R,0} < d_{U,1}$ in the
interior of $\calF$. Between $a+b-1=0$ and $b=1$,
$d_{R,0}=d_{U,1}$ is a single arc, concave up, with right endpoint
at $(\frac{a}{2},\frac{1}{2})$. Between $a+b-1=0$ and $a+2b-2=0$,
the left endpoint is on $y=\frac{1}{2}$ with $0 < x <
\frac{a}{2}$. On $a+2b-2=0$, the left endpoint is
$(0,\frac{1}{2})$. Between $a+2b-2=0$ and $b=1$, the left endpoint
is on $x=0$ with $0 < y < \frac{1}{2}$. In each case, $d_{R,0}$ is
smaller below the arc and $d_{U,1}$ above it.
\end{lemma}

Note  $a+b-1=0$ is where $y=\frac{1}{2}$ is tangent to
$d_{R,0}=d_{U,1}$ at $(\frac{a}{2},\frac{1}{2})$.

From Lemma~\ref{R0R2}, it suffices to consider $d_{R,2}$ for
$b>1$.

\begin{lemma}
\label{R1R2} Assume $b > 1$. For all $(a,b)$,  $d_{R,1}=d_{R,2}$
at $(\frac{a}{2},\frac{1}{2})$. Below $2a-2b+1=0$,
$d_{R,1}\neq d_{R,2}$ elsewhere and $d_{R,1} < d_{R,2}$ in the
interior of $\calF$. Above $2a-2b+1=0$, $d_{R,1}=d_{R,2}$ is the union of
$(\frac{a}{2},\frac{1}{2})$ and a single arc with positive slopes.
The left endpoint of the arc is on $x=0$ with $0 < y <
\frac{1}{2}$ for $(a,b)$ to the left of  $3a^2+(2b+2)a-2b+1=0$,
 is $(0,0)$ for $3a^2+(2b+2)a-2b+1=0$, and is on $y=0$ with $0 < x < \frac{a}{2}$
 to the right of $3a^2+(2b+2)a-2b+1=0$. The right endpoint is on
$x=\frac{a}{2}$ with $0 < y < \frac{1}{2}$ below $a-b+1=0$ and at
$(\frac{a}{2},\frac{1}{2})$ on or above $a-b+1=0$. Above the arc,
$d_{R,1}$ is smaller and below it, $d_{R,2}$ is.\end{lemma}

Note $a-b+1=0$ is where $x=\frac{a}{2}$ is tangent to $d_{R,1}=d_{R,2}$ at $(\frac{a}{2},\frac{1}{2})$.

\begin{lemma}
\label{R1U0} The hyperbola $d_{R,1}=d_{U,0}$ has asymptotes with
slopes $-1\pm\sqrt{2}$.

On $(\sqrt{2}+1)a-b-1=0$ (which includes equivalence classes 14 and 20), the hyperbola
$d_{R,1}=d_{U,0}$ is degenerate and the point of intersection of
its two lines is on $x=0$.
On equivalence class 14, $d_{R,1}=d_{U,0}$ is two line segments,
each meeting $x=0$ at the same point with $0 < y < \frac{1}{2}$.
The other endpoints are on $y=0$ and $y=\frac{1}{2}$ with $0 < x <
\frac{a}{2}$.  On equivalence class 20, $d_{R,1}=d_{U,0}$ is the
line segment with endpoints $(0,\frac{1}{2})$ and
$(\frac{\sqrt{2}-1}{2},0)$. In both cases, to the right of the
segment or segments, $d_{R,1}$ is smaller and to the left,
$d_{U,0}$ is.

To the right of $(\sqrt{2}+1)a-b-1=0$ (which includes equivalence
class 15),
 only the component of the hyperbola $d_{R,1}=d_{U,0}$ to the right
of the point of intersection of the asymptotes passes through
$\calF$ - one could say it is concave right. The lower endpoint of
this arc is on $y=0$ with $0 < x < \frac{a}{2}$ and its upper
endpoint is on $x=\frac{a}{2}$ with $0 < y \leq \frac{1}{2}$ or on
$y=\frac{1}{2}$ with $0 < x \leq \frac{a}{2}$. To the right of
this arc, $d_{R,1}$ is smaller and to the left, $d_{U,0}$ is.

To the left of $(\sqrt{2}+1)a-b-1=0$ (which includes equivalence
classes 12, 13 and 17), the two components of the hyperbola
$d_{R,1}=d_{U,0}$ are above and below the point of intersection of
their asymptotes and are hence concave up and down, respectively.
In $\calF$, when there are two arcs, $d_{R,1}$ is smaller between
them  and $d_{U,0}$ is smaller on the other sides of the arcs.
When there is just the lower arc, $d_{R,1}$ is smaller above it
and $d_{U,0}$ below.

Assume $(a,b)$ is to the right of $(b^2+1)a^2+2b^3a-2b^3-b^2 =0$
with $b > 1$. To the left of, and on, $a^2+(2b+2)a-4b=0$,
$d_{R,1}=d_{U,0}$ is a single arc which is concave down and has
negative slopes. The upper endpoint is on $x=0$ with $0 < y <
\frac{1}{2}$ to the left of $a^2+(2b+2)a-4b=0$ and is
$(0,\frac{1}{2})$ on $a^2+(2b+2)a-4b=0$. The lower endpoint is on
$y=0$ with $0 < x < \frac{a}{2}$. To the right of the arc,
$d_{R,1}$ is smaller and to the left, $d_{U,0}$ is. The only
difference to the right of $a^2+(2b+2)a-4b=0$, is that the upper
endpoint is on $y=\frac{1}{2}$ with $0 < x < \frac{a}{2}$ and the
arc is not necessarily concave down.
\end{lemma}

Note  $(b^2+1)a^2+2b^3a-2b^3-b^2 =0$ is where $d_{R,0}$, $d_{R,1}$,
$d_{U,0}$ have a triple intersection on $x=0$ for $b\geq \frac{2}{3}$.
Other aspects of where $d_{R,1}$ and $d_{U,0}$
are each smaller than the other will be described as needed.

\begin{lemma}
\label{R1U1} For $b \leq 1$, $d_{R,1}=d_{U,1}$ consists of the
line segments $x=0$ and $y=\frac{b}{2}$. Below $y=\frac{b}{2}$,
$d_{R,1}$ is smaller and above it, $d_{U,1}$ is. For $b
> 1$, $d_{R,1}=d_{U,1}$ is just $x=0$ and $d_{R,1} < d_{U,1}$ in
the interior of $\calF$.\end{lemma}

Note $b=1$ is where a subset of $d_{R,1}=d_{U,1}$ coincides
with $y=\frac{1}{2}$.

\begin{lemma}
\label{R1U2}
For all $(a,b)$, $d_{R,1}=d_{U,2}$ at $(\frac{a}{2},\frac{1}{2})$.
For $b \geq 1$ and on or below $a-2b+2=0$, $d_{R,1}=d_{U,2}$ at no other point
or just $(0,\frac{1}{2})$
and
$d_{R,1} < d_{U,2}$ in the interior of $\calF$.
Above $a-2b+2=0$, $d_{R,1}=d_{U,2}$ is the union of $(\frac{a}{2},\frac{1}{2})$ and a single arc with left endpoint on $x=0$
with $0 < y < \frac{1}{2}$. Below $a-b+1=0$, the right endpoint is on $y=\frac{1}{2}$ with $0 < x < \frac{a}{2}$
and on or above $a-b+1=0$ the right endpoint is $(\frac{a}{2},\frac{1}{2})$.
Below the arc, $d_{R,1}$ is smaller and above it, $d_{U,2}$ is.
\end{lemma}

Note $d_{R,1}=d_{U,2}$ at $(0,\frac{1}{2})$ on $a-2b+2=0$ and is tangent to $y=\frac{1}{2}$
at $(\frac{a}{2},\frac{1}{2})$ on $a-b+1=0$.

From Lemma~\ref{R1U1}, it suffices to consider $d_{U,1}$ where $b<
1$.

\begin{lemma}
\label{U0U1} Assume $b < 1$. On $a=1$, $d_{U,0}=d_{U,1}$ does not
pass through the interior of $\calF$ and $d_{U,0}< d_{U,1}$ on the
interior of $\calF$.

Between $(\sqrt{2}+1)a-b-1 = 0$  and $a=1$ (which includes part of
equivalence class 15), $d_{U,0}=d_{U,1}$ is a single arc. One
endpoint is on $y=\frac{1}{2}$ with $0 < x < \frac{a}{2}$ and the
other is on $x=\frac{a}{2}$ with $0 < y < \frac{1}{2}$. To the
right of the arc, $d_{U,1}$ is smaller and to the left, $d_{U,0}$
is.

On $(\sqrt{2}+1)a-b-1 = 0$ (which includes equivalence class 14
and part of 15), $d_{U,0}=d_{U,1}$ consists of two line segments
that meet $x=0$ with $0 < y < \frac{1}{2}$ at the same point. The
right endpoint of one line segment is on $y=\frac{1}{2}$ with $0 <
x < \frac{a}{2}$ and the right endpoint of the other is on
$x=\frac{a}{2}$ with $0 < y < \frac{1}{2}$. Between the line
segments, $d_{U,1}$ is smaller and on the other sides of the
segments,  $d_{U,0}$ is.

Between $a^2+(2b+2)a-4b = 0$ and $(\sqrt{2}+1)a-b-1 = 0$ (which
includes equivalence classes 12 and 13 and part of 15),
$d_{U,0}=d_{U,1}$ is two arcs. They have distinct left endpoints on $x=0$
with $0 < y < \frac{1}{2}$. One has a right endpoint on
$y=\frac{1}{2}$ with $0 < x < \frac{a}{2}$. The other has a right
endpoint on $x=\frac{a}{2}$ with $0 < y < \frac{1}{2}$. Between
the arcs, $d_{U,1}$ is smaller and on the other sides of the arcs,
$d_{U,0}$ is.

Above $a^2+(2b+2)a-4b = 0$ (which includes equivalence class 17),
$d_{U,0}=d_{U,1}$ is a single arc, concave down, with endpoints on
$x=0$ and $x=\frac{a}{2}$ and not meeting $y=0$ or
$y=\frac{1}{2}$. Below the arc, $d_{U,0}$ is smaller and above it,
$d_{U,1}$ is.

\end{lemma}

Note $a=1$ is where $d_{U,0}=d_{U,1}$ at
$(\frac{a}{2},\frac{1}{2})$ and $a^2+(2b+2)a-4b=0$ is where
$d_{U,0}=d_{U,1}$ at $(0,\frac{1}{2})$ for $b\leq 1$.

\section{Eliminating distance functions}
\label{elim2}

In this section, we fix a distance function and then describe
subsets of the $ab$-plank for which there is no $(x,y)$, in the
associated fundamental region $\calF$, for which the given
distance function is smallest. When that is the case, we will say
that the given distance function is not smallest in that subset of
the $ab$-plank.

\begin{prop}
\label{R0R1U1First} Let $(a,b)$ satisfy $a+2b-2>0$ and $b < 1$ in
the $ab$-plank. For all $(a,b)$ above the arc of the ellipse
$a^2+(2b-2)a+3b^2-2b-1 =0$, the labeled $\calF$ for equivalence class 16 in Figure~\ref{TheFs1}  shows
where each of $d_{R,0}, d_{R,1}$ and $d_{U,1}$ is smaller than the
other two. The fundamental region for equivalence class 9 in
Figure~\ref{TheFs1} shows this for $(a,b)$ on or below that
arc.
\end{prop}

Note $a^2+(2b-2)a+3b^2-2b-1 =0$ is where $d_{R,0}, d_{R,1},
d_{U,1}$ have a triple intersection on $x=0$.

\begin{proof}
See Lemmas~\ref{R0R1}, \ref{R0U1} and \ref{R1U1} for where each of
$d_{R,0}$, $d_{R,1}$ and $d_{U,1}$ is smaller than another, in
pairs. If $d_{R,0}=d_{R,1}=d_{U,1}$ on $x=0$ then
$a^2+(2b-2)a+3b^2-2b-1=0$. We can test sample $(a,b)$ above
(respectively below)  the arc of $a^2+(2b-2)a+3b^2-2b-1=0$ in the
$ab$-plank and see that the $y=\frac{b}{2}$ subset of $d_{R,1}=d_{U,1}$ meets $x=0$ above
(respectively below) where $d_{R,0}=d_{R,1}$ does.
\end{proof}

\begin{cor}
The distance function $d_{R,1}$ is not smallest for $(a,b)$ below
the arc of the ellipse $a^2+(2b-2)a+3b^2-2b-1=0$ in the
$ab$-plank. \end{cor}

\begin{proof}
This follows from Lemma~\ref{R0R1} and
Proposition~\ref{R0R1U1First}.
\end{proof}

\begin{prop}
The distance function $d_{R,1}$ is not smallest below the arc of
$a-b^2=0$ in the $ab$-plank.
\end{prop}

Note on $a-b^2=0$ for $a' \leq a \leq 1$ we have
$d_{R,0}=d_{R,1}=d_{U,0}=d_{U,1}$ at a point in $\calF$.

\begin{proof}
From the previous proposition, it suffices to prove this for
$(a,b)$ above $a^2+(2b-2)a+3b^2-2b-1=0$ and below $a-b^2=0$ in the
$ab$-plank. We restrict to such $(a,b)$. From
Proposition~\ref{R0R1U1First}, the region of $\calF$ on which
$d_{R,1}$ could be smallest is bounded above by $y=\frac{b}{2}$,
to the left by $x=0$ and to the right by $d_{R,0}=d_{R,1}$. We
call this Region Left.

Using the results of Lemmas~\ref{R1U0} (note this subset of the $ab$-plank is to the right of
$(\sqrt{2}+1)a-b-1 = 0$)  and \ref{R1U1}, we see that the
region on which $d_{R,1}$ could be smallest is bounded above by
$y=\frac{b}{2}$, to the left by $d_{R,1}=d_{U,0}$, to the right by
$x=\frac{a}{2}$ and below by $y=0$. We call this Region Right.

In order for $d_{R,1}$ to be smallest somewhere, the interiors of Region Left and
Region Right must overlap. Both regions are bounded above by
$y=\frac{b}{2}$. The rightmost point of Region Left is on
$y=\frac{b}{2}$ and has $x$-coordinate
$x_L=\frac{1}{2}\sqrt{a^2+(2b-2)a+3b^2-2b-1}$. The leftmost point
of Region Right, that is on $y=\frac{b}{2}$, has $x$-coordinate
$x_R=\frac{1}{2}\sqrt{a^2+(2b+2)a-(b+1)^2}$.
For $(a,b)$ above $a^2+(2b-2)a+3b^2-2b-1=0$,
the subset below $a=b^2$ is defined by $x_L < x_R$.

The specified arc of $d_{R,0}=d_{R,1}$ bounds Region Left on the
right. Its slope at $(x_L,\frac{b}{2})$ is $\frac{x_L}{-x_L+b}$
(which is positive) and the arc is concave up. The subset of the
$ab$-plank of consideration is between the graphs of
$a^2+(2b-2)a+3b^2-2b-1=0$ (where  $\frac{x_L}{-x_L+b}=0$) and
$a^2+(2b-2)a+2b^2-2b-1=0$ (where  $\frac{x_L}{-x_L+b}=1$). Using
continuity, we see that for all $(a,b)$ of interest we have
$\frac{x_L}{-x_L+b} < 1$. The specified arc of $d_{R,1}=d_{U,0}$
bounds Region Right on the left. Its slope at $(x_R,\frac{b}{2})$
is always 1.
So, given the slopes and concavities (see Lemma~\ref{R1U0}), it is
impossible for Region Left and Region Right to overlap.
\end{proof}

\begin{prop}
\label{WhereR2NotSmallest} The distance function $d_{R,2}$ is not
smallest below the arc of the ellipse $2a^2+(-3b+1)a+2b^2-2b=0$
from $(a,b)=(0,1)$ to $(1,\frac{3}{2})$ in the $ab$-plank.
\end{prop}

Note $2a^2+(-3b+1)a+2b^2-2b=0$ is where $d_{R,0}, d_{R,1},
d_{R,2}$ have a triple intersection on $x=\frac{a}{2}$.

\begin{proof}
For $b\leq 1$, the result follows from Lemma~\ref{R0R2}.  We can test a sample $(a,b)$ below the arc of the
ellipse $2a^2+(-3b+1)a+2b^2-2b=0$ and above $b=1$ to see that
$d_{R,0}=d_{R,2}$ meets $x=\frac{a}{2}$ above where
$d_{R,1}=d_{R,2}$ does. Combining Lemmas~\ref{R0R2} and \ref{R1R2}
with the fact that the arc
of the ellipse $2a^2+(-3b+1)a+2b^2-2b=0$ from $(a,b)=(0,1)$ to
$(1,\frac{3}{2})$ in the $ab$-plank is below $2ab-2a-1=0$ and
below $a-b+1=0$,
gives the result.\end{proof}

\begin{lemma}
\label{R1U0andU0U1}
For $(a,b)$ to the left of $(\sqrt{2}+1)a-b-1 = 0$, the upper components of $d_{R,1}=d_{U,0}$ and of
$d_{U,0}=d_{U,1}$ satisfy $y > \frac{b}{2}$ and their lower components satisfy $y < \frac{b}{2}$.
\end{lemma}

\begin{proof}
It region of the $ab$-plank where $d_{R,1}=d_{U,0}$ and $d_{U,0}=d_{U,1}$ have upper and lower
components is the subset to the left of $(\sqrt{2}+1)a-b-1 = 0$.
At the minimum of the upper components and the maximum of the lower components, the slopes are 0.
On $d_{R,1}=d_{U,0}$ that occurs where $x=y-\frac{b}{2}$ and on $d_{U,0}=d_{U,1}$ that occurs where
$x=-y+\frac{b}{2}$. We use those to replace $x$ in $d_{R,1}=d_{U,0}$ and $d_{U,0}=d_{U,1}$
and solve for $y$ to get $y=\frac{b}{2}\pm \frac{1}{2}\sqrt{-\frac{1}{2}a^2+(-b-1)a + \frac{1}{2}b^2+b+\frac{1}{2}}$
in both cases. Note that $-\frac{1}{2}a^2+(-b-1)a + \frac{1}{2}b^2+b+\frac{1}{2}=0$ on $(\sqrt{2}+1)a-b-1 = 0$.
\end{proof}

\begin{prop}
\label{minU0first} The distance function $d_{U,0}$ is not smallest
above $a^2+2ab-2b=0$ in the $ab$-plank. \end{prop}

\begin{proof} This follows from Lemma~\ref{R0U0}.
\end{proof}

\begin{prop}
\label{minU0} The distance function $d_{U,0}$ is not smallest for
any $(a,b)$ simultaneously above  $(b^2+1)a^2 +2b^3a
-2b^3-b^2=0$ and $a^2+(2b+2)a-4b=0$. \end{prop}

Note $(b^2+1)a^2 +2b^3a -2b^3-b^2=0$ is where $d_{R,0}, d_{U,0},
d_{U,1}$ have a triple intersection on $x=0$ and where
$d_{R,0},d_{R,1}, d_{U,0}$ have a triple intersection on $x=0$.
This does not imply a quadruple intersection on $x=0$ since $x=0$ is
a subset of $d_{R,1}=d_{U,1}$.

\begin{proof} Note, from Proposition~\ref{minU0first} we
can restrict to $(a,b)$ that are also below $a^2+2ab-2b=0$.  In
the $ab$-plank, the line $(\sqrt{2}+1)a-b-1=0$ is always to the
right of $a^2+(2b+2)a-4b=0$ (though they are tangent at $b=1$);
see Lemma~\ref{R1U0}.
A straightforward computation shows for all $(a,b)$ of concern, the lower component of
$d_{R,1}=d_{U,0}$ meets $\calF$ in a single arc, with negative
slopes, meeting $x=0$ with $0 < y < \frac{1}{2}$ and $y=0$ with $0
< x < \frac{a}{2}$.
We can test a sample $(a,b)$ above $(b^2+1)a^2 +2b^3a -2b^3-b^2=0$
to see that $d_{R,0}=d_{U,0}$ meets $x=0$ above where
$d_{R,1}=d_{U,0}$ does; and see Lemma~\ref{R0U0}. So the only
place in $\calF$ where $d_{U,0}$ can be smallest is above the
upper arc of $d_{R,1}=d_{U,0}$.
Note there are $(a,b)$ for which this upper arc does,
and does not pass through $\calF$.

For $(a,b)$ above $a^2+(2b+2)a-4b=0$
Lemmas~\ref{U0U1} and \ref{R1U0andU0U1} show that the region above the upper
arc of $d_{R,1}=d_{U,0}$ does not intersect the region where $d_{U,0} < d_{U,1}$.
\end{proof}

\begin{prop}
 The distance function $d_{U,1}$
is not smallest above $b=1$ or below $a+b-1=0$ in the $ab$-plank.
\end{prop}

\begin{proof}
See Lemmas~\ref{R1U1} and \ref{R0U1}.
\end{proof}

\begin{prop}
The distance function $d_{U,2}$ is not smallest below $a-2b+2=0$
in the $ab$-plank. \end{prop}

Note $a-2b+2=0$ is where $d_{R,1}=d_{U,2}$ at $(0,\frac{1}{2})$.

\begin{proof}
For $b\geq 1$, see Lemma~\ref{R1U2}. For $b < 1$, $d_{R,0}=d_{U,2}$
does not meet $\calF$ and $d_{R,0} < d_{U,2}$.
\end{proof}

\section{The equivalence classes with $b \leq 1$}
\label{blessthan1}

In Table~\ref{47ec} we give a partition of the $ab$-plank into 47
subsets. In this and Section~\ref{bbiggerthan1} we will show that
all $(a,b)$ in a given subset are equivalent to each other, by the
equivalence relation defined in Section 4. We will also
show that the labeled $\calF$, for each $(a,b)$ in the given
subset, is equivalent to the one in Figure~\ref{TheFs1}. It
will only be at the conclusion of this article that we will know
that each of these subsets is actually a distinct equivalence
class, as there is no equivalence between any distinct pair of
labeled $\calF$'s in Figure~\ref{TheFs1}. By abuse of
language, we will refer to these 47 subsets now as equivalence
classes, even though it has not yet been proven that they are.

\subsection{Equivalence classes 1, 2 and 6}

From Table~\ref{Naive},  $d_{R,0}$ is the only distance function
that can be smallest in the interior of $\calF$ on equivalence
class 1. For equivalence classes 2 and 6 only $d_{R,0}$ and
$d_{U,0}$ can be smallest and Lemma~\ref{R0U0} determines the
labeled $\calF$'s.

\subsection{Equivalence classes 7 - 9}

From Table~\ref{Naive}, only  $d_{R,0}$ and $d_{U,1}$ can be
smallest  and Lemma~\ref{R0U1} determines the labeled $\calF$'s.

\subsection{Equivalence classes 3 - 5, 10, 11}

From Table~\ref{Naive}, $d_{R,0}, d_{U,0}$ and $d_{U,1}$ are the
only distance functions that can be smallest.

\begin{lemma}
\label{Equiv4} On equivalence class 4 we have
$\frac{1}{2}\sqrt{a^2+2ab-2b}$ $ = \frac{1}{2}(-a-2b+2)$.
\end{lemma}

Note that $\frac{1}{2}\sqrt{a^2+2ab-2b}$ $ = \frac{1}{2}(-a-2b+2)$
implies $ab-2a+2b^2-3b+2=0$ unconditionally. Also note
 between $a+b-1=0$ and $a+2 b-2=0$ and to the right of
$a^2+2ab-2b=0$ in the $ab$-plank (which includes equivalence class
4) that $\frac{1}{2}\sqrt{a^2+2ab-2b}$ and $\frac{1}{2}(-a-2b+2)$
are the $x$-coordinates of the intersections of $d_{R,0}=d_{U,0}$
and $d_{R,0}=d_{U,1}$ with $y=\frac{1}{2}$ in $\calF$,
respectively.

\begin{proof}
On equivalence class 4 we have $ab-2a+2b^2-3b+2=0$. We can rewrite
this as $(-a-2b+2)^2=a^2+2ab-2b$ or $-a-2b+2=\pm \sqrt{a^2+2ab
-2b}$. On equivalence class 4 we have $-a-2b+2>0$ and so
$\frac{1}{2}\sqrt{a^2+2ab-2b}$ $=\frac{1}{2}(-a-2b+2)$.
\end{proof}

\begin{prop}
\label{Equiv3.4.5} On equivalence classes 3 and 4, the curves
$d_{R,0}=d_{U,0}$, $d_{U,0}=d_{U,1}$ and the left endpoint of
$d_{R,0}=d_{U,1}$  meet $y=\frac{1}{2}$ from left to right and
coincidentally, respectively. On equivalence class 5, the curves
$d_{R,0}=d_{U,0}$, $d_{U,0}=d_{U,1}$ and the left endpoint of
$d_{R,0}=d_{U,1}$  meet $y=\frac{1}{2}$ from right to left, except
that for some $(a,b)$, the left endpoint of $d_{R,0}=d_{U,1}$
meets $x=0$ for $0 < y < \frac{1}{2}$.
\end{prop}

\begin{proof}
We get the results by evaluating $\frac{1}{2}\sqrt{a^2+2ab-2b}$,
$\frac{1}{2}(-a-2b+2)$ and
$\frac{1}{2}(b-1+\sqrt{a^2+(2b+2)a+b^2-6b+1})$ (the $x$-coordinate where
$d_{U,0}=d_{U,1}$ meets $y=\frac{1}{2}$) at sample $(a,b)$ on
either side of equivalence class 4 (see Lemma~\ref{Equiv4}).
Note, for this to change, it would be necessary that
$d_{U,0}=d_{U,1}$ meet either of the other two on $y=\frac{1}{2}$.
But then they would all meet there and that implies
$ab-2a+2b^2-3b+2=0$, which is part of the definition of
equivalence class 4.\end{proof}

\begin{prop}
\label{Equiv5.11.10} On equivalence classes 10 and 11,
$d_{R,0}=d_{U,1}$, $d_{R,0}=d_{U,0}$, and the upper arc of
$d_{U,0}=d_{U,1}$,
 meet $x=0$ from bottom to top, and
coincidentally, respectively. On equivalence class 5, when
$d_{R,0}=d_{U,1}$ meets $x=0$, it does so above where
$d_{R,0}=d_{U,0}$ does. For all $(a,b)$ in these three equivalence
classes, for which $d_{U,0}=d_{U,1}$ is two arcs, the intersection
of the lower arc with $x=0$ is below that of $d_{R,0}=d_{U,0}$ and
$d_{R,0}=d_{U,1}$ (except at the topmost point of equivalence
class 11 where all four coincide on $x=0$).
\end{prop}

\begin{proof}
This is a straightforward proof much like those of
Lemma~\ref{Equiv4} and Proposition~\ref{Equiv3.4.5}.
\end{proof}

\subsubsection{Equivalence classes 3 and 4}

The results of Lemmas~\ref{R0U0} and \ref{R0U1} and
Proposition~\ref{Equiv3.4.5} show that the graphs of
$d_{R,0}=d_{U,0}$ and of $d_{R,0}=d_{U,1}$ in $\calF$ do not
intersect (on equivalence class 3) or intersect at a single point
on $y=\frac{1}{2}$ (on equivalence class 4) and determine the
labeled $\calF$'s.

\subsubsection{Equivalence class 5}

From Lemma~\ref{R0U1} and Propositions~\ref{Equiv3.4.5} and
~\ref{Equiv5.11.10}, we see that $d_{R,0}=d_{U,0}$ and
$d_{R,0}=d_{U,1}$ cross in the interior of $\calF$. Since these
are both components of conics, they can only cross once. The upper
component (or the only component) of $d_{U,0}=d_{U,1}$ meets there
as well. Since there is only one crossing, we see that
$d_{R,0}=d_{U,0}$ and $d_{R,0}=d_{U,1}$ can not meet the lower
component of $d_{U,0}=d_{U,1}$ in the part of equivalence class 5
where there is a lower component. This combined with
Lemmas~\ref{R0U0} and \ref{U0U1} determines the labeled
$\calF$'s.

\subsubsection{Equivalence class 10} \label{equ.cl.10}

The slopes of $d_{U,0}=d_{U,1}$ are given by $\frac{dy}{dx}=\frac{-2x-2y+b}{2x-2y+b}$.
At $x=0$ the slope is 1. The upper arc is concave up. So the slopes of the upper
arc of $d_{U,0}=d_{U,1}$ in $\calF$ are all greater than or equal
to 1. The slopes of $d_{R,0}=d_{U,0}$
are given by $\frac{dy}{dx}=\frac{x}{y}$ and hence are biggest at its point
 $y=\frac{1}{2}$, where $x < \frac{1}{2}$; so all slopes of
  $d_{R,0}=d_{U,0}$ are less than 1.
 Since
$d_{R,0}=d_{U,1}$ is concave up and passes through $(\frac{a}{2},\frac{1}{2})$,
that point is where the slope, which is $\frac{a+b-1}{a+b+1} < 1$  is greatest;
so all slopes of   $d_{R,0}=d_{U,1}$ are less than 1.
These slope conditions, coupled with the result of
Proposition~\ref{Equiv5.11.10}, shows that the upper arc of
$d_{U,0}=d_{U,1}$ does not intersect the other two arcs and we
note, from Lemma~\ref{U0U1}, that $d_{U,0}$ is smallest above this
arc.

The distance function $d_{U,0}$ could only be smallest elsewhere
if the graph of $d_{R,0}=d_{U,0}$ and the lower arc of
$d_{U,0}=d_{U,1}$ through $\calF$ intersected twice. We now show
that those two arcs are on opposite sides of $y=\frac{b}{2}$. The
point with smallest $y$-coordinate on $d_{R,0}=d_{U,0}$
is $(0,\frac{1}{2}\sqrt{-a^2-2ab+2b+1})$. The curve
$\frac{1}{2}\sqrt{-a^2-2ab+2b+1}= \frac{b}{2}$ does not pass
through equivalence class 10. Evaluating at any sample $(a,b)$ in
this equivalence class shows that
$\frac{b}{2}<\frac{1}{2}\sqrt{-a^2-2ab+2b+1}$. The point with
largest $y$-coordinate on the lower arc of $d_{U,0}=d_{U,1}$ is
its point of intersection with $x+y=\frac{b}{2}$, where the slope
of $d_{U,0}=d_{U,1}$ is 0. The $y$-coordinate of that point is
$\frac{1}{2}(b-\sqrt{(-a^2+(-2b-2)a+(b+1)^2)/2})$, which is less
than $\frac{b}{2}$.
 So the only
part of $\calF$ where $d_{U,0}$ is smallest is above the upper
component of $d_{U,0}=d_{U,1}$. Since $d_{R,0}=d_{U,1}$ does not
intersect the upper component of $d_{U,0}=d_{U,1}$, it bounds the
regions where $d_{R,0}$ and $d_{U,1}$ are smallest in $\calF$.
Then Lemma~\ref{R0U1}  determines the labeled
$\calF$'s.

\subsubsection{Equivalence class 11}

The arguments for equivalence class 10 all hold here except that
$d_{U,0}=d_{U,1}=d_{R,0}$ on $x=0$.

\subsection{Equivalence classes 12 - 17}

For equivalence classes 12 - 15 and 17, the only distance
functions that can be smallest are $d_{R,0}$, $d_{R,1}$, $d_{U,0}$
and $d_{U,1}$. For equivalence class 16, the only distance
functions that can be smallest are $d_{R,0}$, $d_{R,1}$ and
$d_{U,1}$. So to determine, for each equivalence class, which
distance function is smallest where on $\calF$, we do the
following. For these equivalence classes,
Proposition~\ref{R0R1U1First} tells us where each of $d_{R,0}$,
$d_{R,1}$ and $d_{U,1}$ is smaller than the other two. Then, for
all but equivalence class 16, we consider where $d_{U,0}$ is
smaller than each of $d_{R,0}$, $d_{R,1}$, and $d_{U,1}$.

Note that $d_{R,1}=d_{U,1}$ consists of two line segments: $x=0$
and $y=\frac{b}{2}$.  Also $d_{R,1}=d_{U,0}$ and $d_{U,0}=d_{U,1}$
meet $x=0$ at the same one or two points.

\begin{prop}
\label{Equiv12.13.17} On equivalence class 12, the intersections
with $x=0$, from highest to lowest, are the upper arcs
 of $d_{U,0}=d_{U,1}$ and of $d_{R,1}=d_{U,0}$  (which coincide),
  $d_{R,0}=d_{U,0}$ and
the subset $y=\frac{b}{2}$ of $d_{R,1}=d_{U,1}$ in some order,
$d_{R,0}=d_{R,1}$, and the lower arcs of $d_{U,0}=d_{U,1}$ and of
$d_{R,1}=d_{U,0}$ (which coincide). On equivalence classes 13 and
17,  the intersections with $x=0$, from highest to lowest, are the
upper arcs of $d_{U,0}=d_{U,1}$ and of $d_{R,1}=d_{U,0}$
(though in part of equivalence class 17, these do not meet $x=0$), the
subset $y=\frac{b}{2}$ of $d_{R,1}=d_{U,1}$, the lower arcs of
$d_{U,0}=d_{U,1}$ and of $d_{R,1}=d_{U,0}$, $d_{R,0}=d_{R,1}$, and
$d_{R,0}=d_{U,0}$.
\end{prop}

\begin{proof}
This is a straightforward computation much like those of
Lemma~\ref{Equiv4} and Proposition~\ref{Equiv3.4.5}.
\end{proof}

\subsubsection{Equivalence class 16}

From Table~\ref{Naive}, only $d_{R,0}, d_{R,1}$ and $d_{U,1}$ can
be smallest. Thus Proposition~\ref{R0R1U1First} determines the
labeled $\calF$'s.

\subsubsection{Equivalence  classes 12 and 13}
\label{ec12}

On equivalence classes 12 and 13, $d_{R,1}=d_{U,0}$ consists of
two arcs, each having left endpoint on $x=0$ with $0 < y <
\frac{1}{2}$. The right endpoints of the upper and lower arcs are
on $y=\frac{1}{2}$ and $y=0$ (respectively) with $0 < x <
\frac{a}{2}$; and see Lemma~\ref{R1U0}.
The slope on $d_{R,1}=d_{U,0}$ is given by $\frac{dy}{dx}=\frac{2x-2y+b}{2x+2y-b}$.
Since the lower arc is concave down, the biggest slope is at $x=0$ where
the slope is $-1$. So the
slopes of the lower arc of $d_{R,1}=d_{U,0}$ are all negative.

We first show that the subset of $\calF$ on which $d_{U,0} <
d_{U,1}$, that is above the upper arc of $d_{U,0}=d_{U,1}$, is
contained in the subset on which $d_{U,0} < d_{R,1}$ above the
upper arc of $d_{R,1}=d_{U,0}$ and is contained in the subset on
which $d_{U,0} < d_{R,0}$ (see Lemmas~\ref{R0U0} and \ref{U0U1}).

We saw in Section~\ref{equ.cl.10} that the slopes of the upper arc of $d_{U,0}=d_{U,1}$ are all greater
than or equal to 1 and the slopes of $d_{R,0}=d_{U,0}$ and the upper arc of $d_{R,1}=d_{U,0}$ are less
than 1. This, and the results of Proposition~\ref{Equiv12.13.17},
show that $d_{R,0}=d_{U,0}$ and the upper arc of $d_{R,1}=d_{U,0}$
do not pass through the region above the upper arc of
$d_{U,0}=d_{U,1}$. So $d_{U,0}$ is smallest above the upper arc of
$d_{U,0}=d_{U,1}$. From
Proposition~\ref{Equiv12.13.17}, the region above the upper arc of
$d_{U,0}=d_{U,1}$ is strictly above $y=\frac{b}{2}$ (where
$d_{R,1}=d_{U,1})$.

From Lemmas~\ref{R0U0}, \ref{R1U0}, \ref{U0U1} and \ref{R1U0andU0U1},
any other points where $d_{U,0}$ is smallest must be contained in
the intersection of the region below the lower arcs of
$d_{U,0}=d_{U,1}$ and of $d_{R,1}=d_{U,0}$ and above
$d_{R,0}=d_{U,0}$.
From Proposition~\ref{Equiv12.13.17}, on equivalence class 12,
$d_{R,0}=d_{U,0}$ meets $x=0$ above where the lower arc of $d_{R,1}=d_{U,0}$
does. Given the slopes of these arcs of $d_{R,0}=d_{U,0}$ (see Lemma~\ref{R0U0}) and
$d_{R,1}=d_{U,0}$, we see that this intersection is empty.

From Proposition~\ref{Equiv12.13.17}, for equivalence class 13, we
see that the $y$-coordinate of the intersections of the lower arcs
of $d_{R,1}=d_{U,0}$ and $d_{U,0}=d_{U,1}$ with $x=0$ is greater
than the $y$-coordinate of the intersection of $d_{R,0}=d_{U,0}$
with $x=0$.
The slope of the lower arc of $d_{U,0}=d_{U,1}$ at $x=0$ is 1 and the right
endpoint is on $x=\frac{a}{2}$ with $0 < y < \frac{1}{2}$.  So the region
below the lower arc of $d_{R,1}=d_{U,0}$ is contained in the
region below the lower arc of $d_{U,0}=d_{U,1}$. From Lemmas
\ref{R0U0} and \ref{U0U1} we see there is a second region in which
$d_{U,0}$ is smallest; it is above $d_{R,0}=d_{U,0}$ and below the
lower arc of $d_{R,1}=d_{U,0}$. From
Proposition~\ref{Equiv12.13.17}, this lower region in which
$d_{U,0}$ is smallest is strictly below $y=\frac{b}{2}$ (where
$d_{R,1}=d_{U,1}$ in the interior of $\calF$).

\subsubsection{Equivalence class 14}

This is a boundary of equivalence class 13. All of the arguments
there hold here except that the arcs of $d_{U,0}=d_{U,1}$,
 $d_{R,1}=d_{U,0}$ and the  $y=\frac{b}{2}$ subset of $d_{R,1}=d_{U,1}$
all coincide on $x=0$ (see Lemmas~\ref{U0U1} and \ref{R1U0}).

\subsubsection{Equivalence class 15}

A straightforward computation shows that on equivalence class 15,
$d_{R,1}=d_{U,0}$ is a single arc,
concave right, with upper and lower endpoints on $y=\frac{1}{2}$
and $y=0$, with $0 < x < \frac{a}{2}$, not meeting $x=0$; and  see Lemma~\ref{R1U0}.

On equivalence class 15 we have $b^2 > a$ so
$\frac{1}{2}\sqrt{a^2+(2b+2)a-b^2-2b-1}$

\noindent
 $<\frac{1}{2}\sqrt{a^2+2ba+b^2-2b-1}$ $<
\frac{1}{2}\sqrt{a^2+(2b-2)a+3b^2-2b-1}$. The latter three
expressions are the $x$-coordinates of the intersection points of
$d_{R,1}=d_{U,0}$, $d_{R,0}=d_{U,0}$ and $d_{R,0}=d_{R,1}$
(respectively) with $y=\frac{b}{2}$ (where $d_{R,1}=d_{U,1}$ in
the interior of $\calF$).

Given the locations of the lower
endpoints of $d_{R,0}=d_{U,0}$ (see Lemma~\ref{R0U0}) and of
$d_{R,1}=d_{U,0}$  and where each of these
meets $y=\frac{b}{2}$, we see that $d_{R,0}=d_{U,0}$ meets
$d_{R,1}=d_{U,0}$, and hence $d_{R,0}=d_{R,1}$, for some $y$-value
$y'$ with  $0 < y' < \frac{b}{2}$.

Since the endpoints of $d_{R,1}=d_{U,0}$ are on $y=0$ and
$y=\frac{1}{2}$, and it is part of a conic, the slope can not be 0.
So $d_{R,1}=d_{U,0}$ crosses $y=\frac{b}{2}$  exactly once and
there is exactly one triple intersection for $d_{R,1}$, $d_{U,0}$,
$d_{U,1}$. So the arc $d_{U,0}=d_{U,1}$ passes through that
crossing and can not cross $d_{R,1}=d_{U,0}$ elsewhere. By
evaluating at any $(a,b)$ on equivalence class 15, we see that
$d_{U,0}=d_{U,1}$ is to the left of $d_{R,1}=d_{U,0}$ for $y
> \frac{b}{2}$ and to the right for $y < \frac{b}{2}$.

This, with Lemmas~\ref{R0U0}, \ref{R1U0} and \ref{U0U1}, shows
that the subset of $\calF$ where $d_{U,0}$ is smallest has a
single component. Its right-hand boundary  is $d_{U,0}=d_{U,1}$
above $y=\frac{b}{2}$, is $d_{R,1}=d_{U,0}$ for $y' \leq y \leq
\frac{b}{2}$, and is $d_{R,0}=d_{U,0}$ for $y < y'$ and above the
$y$-coordinate where $d_{R,0}=d_{U,0}$ meets $x=0$.

\subsubsection{Equivalence class 17}

The argument for equivalence class 13 that $d_{U,0}$ is smallest
below the lower arc of $d_{R,1}=d_{U,0}$  and above
$d_{R,0}=d_{U,0}$ and that this region is strictly
below $y=\frac{b}{2}$ (where $d_{R,1}=d_{U,1}$)
is valid for equivalence class 17 as well. On
equivalence class 17, there are $(a,b)$ for which there is, and is
not, an upper arc of $d_{R,1}=d_{U,0}$; and see Lemma~\ref{R1U0}. The
only other place where $d_{U,0}$ could be smallest would be above
the upper arc of $d_{R,1}=d_{U,0}$, when it exists.
From Lemmas~\ref{U0U1} and \ref{R1U0andU0U1}, we have $d_{U,1}<d_{U,0}$
above the upper arc of $d_{R,1}=d_{U,0}$.
So there is no
where else that $d_{U,0}$ is smallest.

\subsection{Equivalence classes 18 - 21}

These equivalence classes are on $b=1$ and form borders for
equivalence classes 25 - 28. So from Table~\ref{Naive}, only
$d_{R,0}$, $d_{R,1}$ and $d_{U,0}$ can be smallest.

\subsubsection{Equivalence class 18}

From Table~\ref{Naive}, only $d_{R,0}$ and $d_{R,1}$ can be
smallest. So  Lemma~\ref{R0R1} determines the labeled $\calF$'s.

\subsubsection{Equivalence classes 19 - 21}

Equivalence classes 19, 20 and 21 are borders of equivalence class
17, 14 and 15, respectively. The arguments we made there involving
$d_{R,0}$, $d_{R,1}$ and $d_{U,0}$ still hold with the following
exceptions: i) $d_{R,0}=d_{R,1}$ passes through
$(\frac{a}{2},\frac{1}{2})$,  ii) on equivalence class 20, the
intersection point of $d_{R,1}=d_{U,0}$ with $x=0$ is at
$(0,\frac{1}{2})$ (see Lemma~\ref{R1U0}), and iii) on equivalence class
21, $d_{R,1}=d_{U,0}$ meets $y=\frac{1}{2}$ with $0 < x <
\frac{a}{2}$.

\section{Equivalence classes for $b > 1$}
\label{bbiggerthan1}

The distance functions that can be smallest for $b >1$ are
$d_{R,0}$, $d_{R,1}$, $d_{R,2}$, $d_{U,0}$ and $d_{U,2}$. In
Section~\ref{SecR0R1U0}, we determine where each of $d_{R,0}$,
$d_{R,1}$ and $d_{U,0}$ is smaller than the other two. In
Section~\ref{SecR2}, we show that the subset in which $d_{R,2}$ is
smaller than $d_{R,0}$, $d_{R,1}$ and $d_{U,0}$ is contained in
the subset where $d_{R,1}$ is smaller than $d_{U,0}$. We use this
to show where $d_{R,2}$ is smaller than  $d_{R,0}$, $d_{R,1}$ and
$d_{U,0}$. In Section~\ref{SecU2}, we show that the subset in
which $d_{U,2}$ is smaller than the other four is contained in the
subset where $d_{R,1}$ is smaller than $d_{R,0}$, $d_{R,2}$ and
$d_{U,0}$. Lastly we determine where $d_{U,2}$ is smaller
than $d_{R,1}$. This more holistic approach enables us to
determine the equivalence class of each $(a,b)$, for $b>1$,
without needing to break this section into a subsection for each
equivalence class.

\subsection{The distance functions $d_{R,0}$, $d_{R,1}$ and
$d_{U,0}$}\label{SecR0R1U0}

In Proposition~\ref{minU0}, we showed that to the left of
$(b^2+1)a^2+2b^3a-2b^3-b^2 = 0$, the distance function $d_{U,0}$
can not be smallest. We then use Lemma~\ref{R0R1} to see where
each of $d_{R,0}$ and $d_{R,1}$ is smaller than the other. On
equivalence class 25, only $d_{R,0}$ and $d_{R,1}$ can be smallest.
The labeled $\calF$'s for this equivalence
class are determined by Lemma~\ref{R0R1} since this equivalence class is to the lower left of
$(4b-2)a-2b-1=0$.

For the remainder of Section~\ref{SecR0R1U0}, we restrict to
$(a,b)$ to the right of $(b^2+1)a^2+2b^3a-2b^3-b^2 = 0$. See
Lemmas \ref{R0R1}, \ref{R0U0} and \ref{R1U0} for where each
of $d_{R,0}$, $d_{R,1}$ and $d_{U,0}$ is smaller than each other,
in pairs.

\begin{prop}
\label{R0R1U0inF} For all $(a,b)$ to the right of
$(b^2+1)a^2+2b^3a-2b^3-b^2 = 0$ there is exactly one triple
intersection for $d_{R,0}$, $d_{R,1}$, $d_{U,0}$; it is in the
interior of $\calF$.
\end{prop}

\begin{proof}
To the right of
$a^2+(2b+2)a-4b=0$ the result follows from Lemmas~\ref{R0R1} and
\ref{R1U0}.
To the left of, and on $a^2+(2b+2)a-4b=0$,  the $y$-intercept of
$d_{R,1}=d_{U,0}$ on $x=0$ is
$y=\frac{b-\sqrt{-a^2-2ab-2a+b^2+2b+1}}{2}$ and the $y$-intercept
of $d_{R,0}=d_{U,0}$ is $y=\frac{\sqrt{-a^2-2ab+2b+1}}{2}$
(note  $(b^2+1)a^2+2b^3a-2b^3-b^2 = 0$ is to the right of
$a^2+2ab-2b=0$ - see Lemma~\ref{R0U0}).
The
former is greater than the latter if and only if $(a,b)$ is to the
right of $(b^2+1)a^2+2b^3a-2b^3-b^2 = 0$.
\end{proof}

We see $d_{U,0}$ is smaller than the other two to the left of
$d_{R,1}=d_{U,0}$ and above $d_{R,0}=d_{U,0}$. In a neighborhood
of $(\frac{a}{2},\frac{1}{2})$, we see $d_{R,1}$ is smaller than
the other two.

Now we need to know how each of the curves in $\calF$, at which
two of $d_{R,0}$, $d_{R,1}$ and $d_{U,0}$ are the same, meet the
borders of $\calF$. Considering only where each of $d_{R,0}$,
$d_{R,1}$ and $d_{U,0}$ is smaller than the other two, let us
consider the possible labeled $\calF$'s  for $(a,b)$ to the right
of $(b^2+1)a^2+2b^3a-2b^3-b^2 = 0$. From
Proposition~\ref{R0R1U0inF}, the equivalence class of a labeled
$\calF$  can only change if one of the four possibilities in
Lemma 6.1  occurs.

For $d_{R,0}=d_{U,0}$, the only one of these that occurs to the
right of $(b^2+1)a^2+2b^3a-2b^3-b^2 = 0$ is that $d_{R,0}=d_{U,0}$
at $(0,0)$ and $(\frac{1}{2},\frac{1}{2})$ on $a=1$. For
$d_{R,1}=d_{U,0}$, the only one of these that occurs is that
$d_{R,1}=d_{U,0}$ at $(0,\frac{1}{2})$ on $a^2+(2b+2)a-4b = 0$.
For $d_{R,0}=d_{R,1}$, the only one of these that occurs is that
$d_{R,0}=d_{R,1}$ at $(\frac{a}{2},0)$ on $(4b-2)a-2b-1=0$.

We now show that having $d_{R,0}=d_{R,1}$ at $(\frac{a}{2},0)$
does not lead to a change
in equivalence classes for $a < 1$.
Let us consider the $(a,b)$ for which $a
<1$ and that are to the right of $2ab-2a-1=0$; note this subset of the $ab$-plank
contains $(4b-2)a-2b-1=0$ for $a < 1$. For $(a,b)$ in this
subset,  Lemma~\ref{R0R2} shows that $d_{R,2}$ is
smaller than $d_{R,0}$ in a neighborhood of
$(\frac{a}{2},0)$. So having
$d_{R,0}=d_{R,1}$ at $(\frac{a}{2},0)$ does not lead to a change in equivalence class
along $(4b-2)a-2b-1=0$ for $a < 1$.

The only  distance functions that can be smallest on equivalence
classes 26 - 29 and 32  are $d_{R,0}$, $d_{R,1}$ and $d_{U,0}$.
The labeled $\calF$'s, considering only  $d_{R,0}$, $d_{R,1}$ and $d_{U,0}$,
are determined by the previous paragraph and
Lemmas~\ref{R0R1}, \ref{R0U0} and \ref{R1U0}.

\subsection{Where $d_{R,2}$ is smaller than $d_{R,0}$, $d_{R,1}$ and
$d_{U,0}$} \label{SecR2}

From Proposition~\ref{WhereR2NotSmallest},  $d_{R,2}$ can only be
smallest above $2a^2+(-3b+1)a+2b^2-2b=0$.
For the rest of Section~\ref{SecR2}, we restrict to that subset of
the $ab$-plank.

Note $d_{U,0}$ can be smaller than $d_{R,1}$ and $d_{R,0}$
only to the right of $(b^2+1)a^2+2b^3a-2b^3-b^2=0$.

\begin{prop}
\label{R2U0R1} Let $(a,b)$ be above $2a^2+(-3b+1)a+2b^2-2b=0$ and
to the right of $(b^2+1)a^2+2b^3a-2b^3-b^2=0$.
The closures of the regions on which $d_{R,2}$ is smallest and on
which $d_{U,0}$ is smallest are disjoint.
\end{prop}

\begin{proof}
See Lemmas~\ref{R1R2} and \ref{R1U0} for a description of where
each distance function in the pairs, $d_{R,1}$, $d_{R,2}$ and
$d_{R,1}, d_{U,0}$ is smaller than the other.
If $d_{R,1}=d_{R,2}=d_{U,0}$ on $y=0$ then $a^2+b^2=0$.
We test a sample $(a,b)$ in this region and find that $d_{R,1}=d_{U,0}$ meets
$y=0$ to the left of $d_{R,1}=d_{R,2}$ so this is true for all $(a,b)$.

As long as
$d_{R,1}=d_{U,0}$ and $d_{R,1}=d_{R,2}$ do not cross twice in the
interior of $\calF$, then the result follows. We choose a sample
$(a,b)$ in these equivalence classes and find that
$d_{R,1}=d_{U,0}$ and $d_{R,1}=d_{R,2}$ do not cross twice in
$\calF$ for that $(a,b)$. If there is an $(a,b)$ in these
equivalence classes for which they cross twice, then there would
be a spontaneous generation of a triple intersection for
$d_{R,1}$, $d_{R,2}$, $d_{U,0}$; but this can not occur from
Theorem 8.1 .
\end{proof}

We again consider the entire subset of the $ab$-plank on which
$d_{R,2}$ can be smallest.
From
Proposition~\ref{R2U0R1}, it suffices to determine where $d_{R,2}$
is smaller than $d_{R,0}$ and $d_{R,1}$. Next we describe triple
intersections for $d_{R,0}$, $d_{R,1}$, $d_{R,2}$.

\begin{prop}
\label{R0R1R2} There is no triple intersection for $d_{R,0}$,
$d_{R,1}$, $d_{R,2}$ on $x=0$ or $y=\frac{1}{2}$. There is a
triple intersection  on $y=0$ for $(a,b)$ with
$(2b-2)a^3+(b^2-1)a^2-b^2 = 0$
and on $x=\frac{a}{2}$  for $(a,b)$ with $2a^2+(-3b+1)a+2b^2-2b=0$
(which is the border of where $d_{R,2}$ can be smallest). There is
a triple intersection in the interior of $\calF$ if and only if
$(a,b)$ is between those two curves.
\end{prop}

\begin{proof}
Straightforward computations show for which $(a,b)$ the triple
intersection is on $y=0$ and on $x=\frac{a}{2}$ and that none
occur on $x=0$ and on $y=\frac{1}{2}$. To show the result about
the triple intersection in the interior, we combine the
information from Lemmas~\ref{R0R1} and \ref{R0R2} with the
following facts. Above $2a^2+(-3b+1)a+2b^2-2b=0$ and on or below
$2ab-2a-1=0$, the right endpoint of $d_{R,0}=d_{R,1}$ meets
$x=\frac{a}{2}$ above where the right endpoint of
$d_{R,0}=d_{R,2}$ does. No extra information is needed for $(a,b)$
above $2ab-2a-1=0$ and to the left of or on $(4b-2)a-2b-1=0$. To
the right of $(4b-2)a-2b-1=0$ and to the left of
$(2b-2)a^3+(b^2-1)a^2-b^2 = 0$, $d_{R,0}=d_{R,1}$ meets $y=0$ to
the right of $d_{R,0}=d_{R,2}$. To the right of
$(2b-2)a^3+(b^2-1)a^2-b^2 = 0$, $d_{R,0}=d_{R,1}$ meets $y=0$ to
the left of $d_{R,0}=d_{R,2}$ and their continuations meet below
$y=0$; so the two arcs can not meet in the interior of $\calF$.
\end{proof}

Considering only where each of $d_{R,0}$, $d_{R,1}$, $d_{R,2}$ and
$d_{U,0}$ is smaller than the other three, let us consider the
possible equivalence classes for labeled $\calF$'s. We have the
information regarding where each of  $d_{R,0}$, $d_{R,1}$, and
$d_{U,0}$ is smaller than the other two from
Section~\ref{SecR0R1U0}. From Proposition~\ref{R2U0R1} the regions
where $d_{R,2}$ and $d_{U,0}$ can be smallest are disjoint. From
Proposition~\ref{R0R1R2}, a change in equivalence class from a
triple intersection for $d_{R,0}$, $d_{R,1}$, $d_{R,2}$ occurs
only on $2a^2+(-3b+1)a+2b^2-2b=0$ where the triple intersection
meets $x=\frac{a}{2}$ with $0 \leq y < \frac{1}{2}$ and on
$(2b-2)a^3+(b^2-1)a^2-b^2 = 0$ where the triple intersection meets
$y=0$ with $0 < x \leq \frac{a}{2}$.
The only other changes in equivalence class occur from one of the
four possibilities described in Lemma 6.1  for
the curve $d_\alpha=d_{R,2}$ with $d_\alpha \in \{ d_{R,0},
d_{R,1}\}$.

For $d_{R,0}=d_{R,2}$, the only
one of the four possibilities from Lemma 6.1  that occurs is that $d_{R,0}=d_{R,2}$ passes through
$(\frac{a}{2},0)$ on $2ab-2a-1=0$.
For $d_{R,1}=d_{R,2}$, there are two that occur. The right endpoint is
$(\frac{a}{2},\frac{1}{2})$ on $a-b+1=0$. Having $d_{R,1}=d_{R,2}$ at $(0,0)$ (which occurs along
$3a^2+(2b+2)a-2b+1=0$) does not lead to a change in equivalence class;
Lemma~\ref{R0R2} shows that $d_{R,0}$ is smaller than $d_{R,2}$ in a neighborhood of
$(0,0)$ for all $b > 1$.

For $b > 1$, we combine the earlier results about where each of $d_{R,0}$, $d_{R,1}$ and $d_{U,0}$ is smaller than
the other two with the results of this section and those of
 Lemmas~\ref{R0R2} and \ref{R1R2}. Then, considering only $d_{R,0}$, $d_{R,1}$, $d_{R,2}$ and $d_{U,0}$,
 we have determined the labeled $\calF$'s for all $b > 1$.

\subsection{Where $d_{U,2}$ is smaller than the others} \label{SecU2}

The remaining equivalence classes for which the labeled $\calF$'s
need to be determined are those in which $d_{U,2}$ can be
smallest.

\begin{prop}
\label{U2smallest} The closure of where $d_{U,2}$ is smallest
intersects the closures of where $d_{R,0}$, $d_{R,2}$ and
$d_{U,0}$ are smallest only at $(\frac{a}{2},\frac{1}{2})$ or
nowhere.
\end{prop}

\begin{proof}
We first show this for $d_{R,0}$. We have $d_{R,0}=d_{U,2}$ is the
horizontal line segment $y=\frac{a+b+1-ab}{2b+2}$.
Since $b > 1$ we have $\frac{a+b+1-ab}{2b+2} = \frac{1}{2} - \frac{a(b-1)}{2b+2} < \frac{1}{2}$.
Also $0=\frac{a+b+1-ab}{2b+2}$ does not pass through the $ab$-plank.
So we
have $0 < \frac{a+b+1-ab}{2b+2}< \frac{1}{2}$. Below the line
segment, $d_{R,0}$ is smaller and above it, $d_{U,2}$ is. See
Lemma~\ref{R0R1} for a description of where each of $d_{R,0}$ and
$d_{R,1}$ is smaller than the other. It suffices to show that $
d_{R,0}=d_{R,1}$ is below $d_{R,0}=d_{U,2}$.

The slope of $d_{R,0}=d_{R,1}$ is given by $\frac{dy}{dx}=\frac{2x+2y-b}{-2x+2y+b}$.
Since $b>1$ and $y\leq \frac{1}{2}$, the slope at the left endpoint of $d_{R,0}=d_{R,1}$ (where $x=0$)
is negative.
On $a+b-\sqrt{2}-1=0$, the hyperbola $d_{R,0}=d_{R,1}$ is degenerate
and $d_{R,0}=d_{R,1}$ passes through $\calF$ as a line segment with
negative slopes. Above
$a+b-\sqrt{2}-1=0$ the two components of the hyperbola are to the
left and right of the point of intersection of the asymptotes.
The component to the left passes through $\calF$. The point on the
component with infinite slope is on $y=x-\frac{b}{2}$, which does
not pass through $\calF$ since $x \leq \frac{1}{2}$ and $b > 1$. Testing any sample $(a,b)$ we
see that the arc of $d_{R,0}=d_{R,1}$ in $\calF$
is concave down or straight and all slopes on
the arc are negative; for $(a,b)$ below $a+b-\sqrt{2}-1=0$,
$d_{R,0}=d_{R,1}$ is concave up. When the $y$-coordinate of the
right endpoint is greater than 0, the $y$-coordinates of the left
and the right endpoints of $d_{R,0}=d_{R,1}$ are the same on
$5a^2-8a+12b^2-8b-4=0$. But this ellipse is strictly below the
equivalence classes where $d_{U,2}$ can be smallest. So in these
equivalence classes, the left endpoint is always above the right
endpoint. In all cases, it suffices to show that the intersection
of $d_{R,0}=d_{R,1}$ with $x=0$ is below that of
$d_{R,0}=d_{U,2}$. If $d_{R,0}=d_{R,1}=d_{U,2}$ on $x=0$ then $(b^2+1)a=0$.
Testing any sample $(a,b)$ we find the result follows.

We now show the result for $d_{R,2}$. We note $d_{R,2}=d_{U,2}$ is
a line segment, of slope 1, with right endpoint
$(\frac{a}{2},\frac{1}{2})$; below the segment,  $d_{R,2}$ is
smaller and above it, $d_{U,2}$ is. Also, $d_{R,1}=d_{U,2}$ is a
single arc, with non-negative slopes, with left endpoint on $x=0$ for $0 < y < \frac{1}{2}$
and right endpoint on $y=\frac{1}{2}$ with $0 < x \leq
\frac{a}{2}$. Above the arc, $d_{U,2}$ is smaller and below it,
$d_{R,1}$ is. Above $a^2-2b=0$ (which includes the $(a,b)$ of
concern), $d_{R,1}=d_{U,2}$ meets $x=0$ above where
$d_{R,2}=d_{U,2}$ does. We test a sample $(a,b)$ above $a-2b+2=0$
(where $d_{U,2}$ can be smallest) and find that $d_{R,2}=d_{U,2}$
is below $d_{R,1}=d_{U,2}$. For this to change, there would need
to be a spontaneous generation of a triple intersection, which
does not occur from Theorem 8.1 .

Lastly we show the result for $d_{U,0}$. We need only consider the
equivalence classes in which $d_{R,1}$, $d_{U,0}$ and $d_{U,2}$
can be smallest (they are to the right of
$(b^2+1)a^2+2b^3a-2b^3-b^2=0$ and above $a-2b+2=0$). See
Lemma~\ref{R1U0} and the previous paragraph for descriptions of
where each of $d_{R,1}$ and $d_{U,0}$ and each of $d_{R,1}$ and
$d_{U,2}$ is smaller than the other, respectively. So it suffices
to show that the intersection of $d_{R,1}=d_{U,0}$ with $x=0$ is
below that of $d_{R,1}=d_{U,2}$. If we have a triple intersection
for $d_{R,1}$, $d_{U,0}$, $d_{U,2}$ on $x=0$ then $(a,b)$ is on
$(b^2+1)a^2+2a-b^2-2b=0$, which does not pass through these
equivalence classes. We pick one sample $(a,b)$ in these
equivalence classes and see that indeed, the intersection of
$d_{R,1}=d_{U,0}$ with $x=0$ is below that of $d_{R,1}=d_{U,2}$.
\end{proof}

From Proposition~\ref{U2smallest}, the only new changes in
equivalence class, (not coming from where each of $d_{R,0}$,
$d_{R,1}$, $d_{R,2}$ and $d_{U,0}$ are smaller than the other
three) come from how $d_{R,1}=d_{U,2}$ meets the boundary of
$\calF$. The only change of the four possibilities described in
Lemma 6.1  that occurs is that on and above
$a-b+1=0$, the right endpoint of $d_{R,1}=d_{U,2}$ is
$(\frac{a}{2},\frac{1}{2})$.

For $b > 1$, we combine the earlier results about where each of $d_{R,0}$, $d_{R,1}$, $d_{R,2}$ and $d_{U,0}$ is smaller than
the other three with the results of this section and those of
Lemma~\ref{R1U2}. Together these prove
that the labeled $\calF$'s are as in Figure~\ref{TheFs2} .

\section{Conclusion of Appendix}

In Section~\ref{blessthan1}, we showed for each $(a,b)$ in one of
the equivalence classes 1 - 21, that the labeled $\calF$ in
Figure~\ref{TheFs1}
 is correct. In Section~\ref{bbiggerthan1},
we showed for each $(a,b)$ in one of the equivalence classes 22 -
47, that the labeled $\calF$ in Figure~\ref{TheFs2}  is
correct. Since no two of the 47 labeled $\calF$'s are equivalent,
this finally proves that each of the equivalence classes listed in
Table~\ref{47ec}  actually is an equivalence class.

\end{document}